\documentclass[11pt, reqno, letterpaper]{amsart}

\usepackage{esint}
\usepackage{amssymb}

\usepackage[%pagebackref,
hypertexnames=false, colorlinks, 
linkcolor=black,citecolor=black,urlcolor=black, linktocpage=true
]%
{hyperref} %,hypertexnames=false,colorlinks,[pagebackref]
\hypersetup{bookmarksdepth=3}

	\normalbaselines						  %

\usepackage{graphicx}
\usepackage{%stmaryrd,
mathrsfs}

\usepackage[T1]{fontenc}

\newtheorem{thm}{Theorem}[section]
\newtheorem*{thmm}{Theorem}
\newtheorem*{corr}{Corollary}

\newtheorem{cor}[thm]{Corollary}

\theoremstyle{definition}

\newtheorem*{df*}{Definition}

\theoremstyle{remark}

\newtheorem*{rem*}{Remark}

\numberwithin{equation}{section}

\newtheorem{theorem}{Theorem}
\newtheorem{problem}{Problem}
\newtheorem{lemma}{Lemma}

\newtheorem{remark}{Remark}

\newcommand{\La}{\langle}
\newcommand{\Ra}{\rangle}

\newcommand{\B}{\mathbb{B}}
\newcommand{\R}{\mathbb{R}}

\newcommand{\E}{\mathbb{E}}
\newcommand{\tr}{{ \mathrm{tr}}}
\newcommand{\Hess}{{ \,\mathrm{Hess}\,}}
\newcommand{\Ker}{{ \mathrm{ker}\,}}

\newcommand{\eps}{\varepsilon}

\usepackage{graphicx}
\DeclareGraphicsExtensions{.eps}

\newcommand{\s}{\mathbf}

\begin{document}

\title[The hill property]{Bellman partial differential equation and the hill property for classical isoperimetric problems}

\author{Paata Ivanisvili}
\thanks{PI is partially supported by the Hausdorff Institute for Mathematics, Bonn, Germany}
\address{Department of Mathematics,  Michigan State University, East
Lansing, MI 48824, USA}
\email{ivanishvili.paata@gmail.com}

\author{Alexander Volberg}
\thanks{AV is partially supported by the NSF grant DMS-1265549 and by the Hausdorff Institute for Mathematics, Bonn, Germany}
\address{Department of Mathematics,  Michigan State University, East
Lansing, MI 48824, USA}
\email{volberg@math.msu.edu}

\makeatletter
\@namedef{subjclassname@2010}{
  \textup{2010} Mathematics Subject Classification}
\makeatother


% 42B	Harmonic analysis in several variables
% 42B20	Singular and oscillatory integrals (Calder?on-Zygmund, etc.)
% 42B35	Function spaces arising in harmonic analysis

% 47A	General theory of linear operators
% 47A30	Norms (inequalities, more than one norm, etc.)

%{30E20, 47B37, 47B40, 30D55.}
%
% 30D55	$H^p$-classes (1980-2009)
% 30E20	Integration, integrals of Cauchy type, integral representations of analytic functions
%
% 47B   	Special classes of linear operators
% 47B37	Operators on special spaces (weighted shifts, operators on sequence spaces, etc.)
% 47B40	Spectral operators, decomposable operators, well-bounded operators, etc.


  \begin{abstract}
The goal of this note is to have a systematic approach to generating isoperimetric inequalities from two concrete type of PDEs. We call these PDEs Bellman type because a totally analogous equations happen to rule many sharp estimates for singular integrals in harmonic analysis, and such estimates were obtained with the use of Hamilton--Jacobi--Bellman PDE.  We show how classical inequalities of Brascamp--Lieb, Prekopa--Leindler, Ehrhard are particular case of this scheme, which allows us to augment the stock of such inequalities. We approach the isoperimetric inequalities as a maximum (minimum) principle for special types of functions. These functions are compositions of ``Bellman function" and an appropriate flow built on test functions. Then the existence of maximum (minimum) principle for such compositions can be reduced to the  requirement that  Bellman function satisfies a concrete class of  nonlinear PDE (written down below). We are left to solve this nonlinear PDE (sometimes a possible task) to enjoy isoperimetric inequalities. The nonlinear PDE that we will describe in this article can be reduced sometimes to solving Laplacian eigenvalue problem,
$\bar{\partial}$-equation of certain type  or just the linear heat equation.
\end{abstract}

\date{}
\maketitle

\setcounter{equation}{0}
\setcounter{theorem}{0}

\section{Introduction: what kind of Bellman PDE we consider here.}
\label{Bell}

The  papers of Ledoux \cite{Ledoux2014}, Barthe \cite{barthe1} and our earlier paper \cite{IvVo} served us as a guide and inspiration for the present article.

In what follows the letter $B$ always stands for a function of $n$ real variables given in some domain $\Omega\subset \R^n$ and satisfying two different but related PDEs. We will describe now these PDEs, they will depend on the choice of matrix $A=(a_{1},\ldots, a_{n})$ of size $k\times n$, where $k\le n$,  $a_m\in \R^k$ is $m$-th column vector  of $A$. Both types of PDEs we are interested in here will also depend on a given symmetric real matrix $C$ of size $k\times k$. Practically always this $C$ will be assumed to be positive: $C>0$ unless we say otherwise, but in fact there are situations where one does not need even nonnegativity, only symmetry would suffice. In this article we assume however $C>0$, but the reader may consult \cite{IvVo}, where one considers arbitrary symmetric $C$'s.

For two matrices $M_1, M_2$ of the same size $M_1\bullet M_2$ denotes the Schur product of them, that is entrywise product. Further $a\cdot b$, or $\langle a, b\rangle$ denotes scalar product of vectors in $a,b \in \mathbb{R}^{m}$. 

We will also need several semigroups. If $C\in M_{k\times k}$ and positive, then operator 
$$
L=L_C:= \sum_{i, j=1}^k c_{ij} \frac{\partial^2}{\partial x_i\partial x_j}
$$ is a negative generator of the semigroup
$$
P_t^C = e^{tL_C}\,.
$$
For a nice function $f(x)$ on $\R^k$ the solution of ``modified heat equation"
$$
\frac{\partial u(x, t)}{\partial t} = (L_C u)(x, t),\,\, u(x, 0)=f(x)
$$
will be denoted by $e^{tL_C} f$ and it is quite easy to see how to construct such a solution.
Consider $C=I_k$ (identity in $\R^k$) and the usual laplacian $\Delta=L_{I_k}$. Consider the solution 
$v(x, t)$ of the usual heat equation:
$$
\frac{\partial v(x, t)}{\partial t} = (\Delta v)(x, t),\,\, v(x, 0)=f(C^{1/2}x)\,.
$$
Then put $u(x, t)= v(C^{-1/2} x, t)$. It solves a modified heat equation. In fact,
the Hessian (in variables $x$) of $v(C^{-1/2}x)$ is 
$C^{-1/2} (\Hess v)(C^{-1/2}x) C^{-1/2}$, and 
\begin{align*}
u(x,t)=\int_{\mathbb{R}^{k}}f(y)p^{C}_{t}(x,y)dy\quad \text{where}\quad p_{t}^{C}(x,y)=\frac{\det(C^{-1/2})}{(4\pi t)^{k/2}}e^{-\frac{|C^{-1/2}(x-y)|^{2}}{4t}}.
\end{align*}

Therefore, 
\begin{align*}
&(L_C u)(x,t)= \tr (C\Hess u)(x, t) = \tr (C\,C^{-1/2} (\Hess v)(C^{-1/2}x, t) C^{-1/2})=\\
&(\Delta v) (C^{-1/2}x, t) = \frac{\partial v}{\partial t}(C^{-1/2}x, t) = \frac{\partial u}{\partial t}(x, t)\,.
\end{align*}

\subsection{Special initial data}
\label{spid}
We will use very often $P_t^C f$ for $f$ having a special form
$$
f(x) := F(a\cdot x)\,,
$$
where $F$ is a function of one variable, $x\in \R^k$, and $a\in \R^k$ is a fixed vector. Then the flow $P_t^C f$  can be constructed as follows. Consider $1D$ heat flow (slightly modified):
$$
\frac{\partial U}{\partial t}(y, t) =\La C a, a\Ra U''(y, t)\,,\,\,\, y\in \R\,,\,\, U(y, 0) =F(y)\,.
$$
Then it turns out that 
\begin{equation}
\label{ax}
P_t^C f (x)= U(a\cdot x, t)=\int_{\mathbb{R}}F(a\cdot x + y\sqrt{2t\langle Ca,a\rangle})d\gamma_{1}(y), \,\,x\in \R^k\,.
\end{equation}
This is of course a simple direct calculation. Notice that if the matrix $C$ is symmetric and $\langle Ca,a\rangle >0$  but not necessarily nonnegative we can rewrite this as follows

\begin{equation}
\label{LC}
\left(\sum_{i, j=1}^k c_{ij} \frac{\partial^{2}}{\partial x_i \partial x_j} -\frac{\partial}{\partial t}\right)U(a\cdot x, t) = \left( \La Ca, a\Ra \frac{\partial^2}{\partial y^2} -\frac{\partial}{\partial t}\right) U (a\cdot x, t)=0\,.
\end{equation}

So in order to construct the flow $P_{t}^{C}f(x)$ for this special initial data $f(x)=F(a\cdot x)$ we do not need $C$ to be positive. We only need symmetric $C$ such that $\langle Ca,a\rangle >0$.

 One more nice property of the special initial data is that 
\begin{align*}
\nabla P_{t}^{C}f(x)=a U'(a\cdot x,t)
\end{align*}
where $U'(a\cdot x, t)= \frac{\partial }{\partial y} U(y,t)|_{y=a\cdot x}$.

For simplicity we work with {\em rank-1 case}. Here {\em rank-1 case} means that we consider initial datas of the form $F(a\cdot x)$, and  $a\cdot x$ corresponds to rank-1 linear operator. One can consider {\em general rank case} i.e., initial data of the form $f(x)=F(xA)$ for some $k\times m$ matrix $A$ and $x \in \mathbb{R}^{k}$. For more details we refer the reader to Section~\ref{grank} where we do consider the  {\em general rank case}.

\subsection{The first type of Bellman PDE: modified concavity property}
\label{I}

Recall that $A, C$, $C>0$ are fixed matrices of size $k\times n$ and $k\times k$ correspondingly, and $B$ is a certain (smooth) function given in $\Omega\subset \R^n$. Here is our first PDE, which can be called 
``modified concavity". We also  assume, unless it is said otherwise,
 that $k\times n$ matrix $A$ has full rank $k$.
\begin{equation}
\label{MC}
\begin{cases}
A^* CA\bullet \Hess  B (x)\le 0\,,\,\, \forall x\in \Omega\\
\det (A^* CA\bullet\Hess B)(x) =0\,,\,\, \forall x\in \Omega\,.
\end{cases}
\end{equation}
Notice that the first line above is a partial differential inequality (not equation): it is a negative  definiteness of a modified Hessian. But we wish to consider the whole system \eqref{MC}, for brevity we call it our first Bellman PDE.

\bigskip

We will show below how such $B$'s provide us with important occurrences of isoperimetric inequalities such as Borell's Gaussian noise stability equation, hypercontractivity of Ornstein--Uhlenbeck semigroup, Brascamp--Lieb Gaussian inequalities.
\bigskip

Condition (\ref{MC}) (the first inequality) implicitly appears in \cite{CLM} for the concrete function $B(x_{1},x_{2},\ldots, x_{n})=x_{1}^{1/p_{1}}\cdots x_{n}^{1/p_{n}}$ and some spacial matrix $C$. See also~\cite{Paouris1}. For further details we refer the reader to \cite{IvVo}. 
\bigskip

But there are instances of very important isoperimetric inequalities, for which the second type of Bellman PDE should be used. It is \eqref{MMC} and it gives a bigger amount of Bellman function $B'$s.

\subsection{The second type of Bellman PDE: modified concavity property}
\label{II}

Along with fixed matrices $A, C$ as above, we need the notation $D$. It is an $n\times n$ non-constant diagonal matrix, where on place $(m, m), m=1, \dots, n,$ we have $B_m:=\frac{\partial B}{\partial x_m}$.
We also assume always that $k\times n$ matrix $A$ has full rank $k$. Here is the second type of Bellman equation, which we assume to hold $\forall x\in \Omega$:
\begin{equation}
\label{MMC}
\begin{cases}
(I-  (AD)^*(AD^2A^*)^{-1}AD)  (A^* CA\bullet \Hess B)(I-  (AD)^*(AD^2A^*)^{-1}AD)  (x)\le 0\,,\\
\det_{n-k}(I-  (AD)^*(AD^2A^*)^{-1}AD)  (A^* CA\bullet \Hess B)(I-  (AD)^*(AD^2A^*)^{-1}AD) (x) =0\,.
\end{cases}
\end{equation}

Here $\det_s$ mean $s\times s$ minors of corresponding matrix-function.  Here instead of $C>0$ we only require $C\geq 0$ and $\langle Ca_{j},a_{j}\rangle >0$ for all $j=1,\ldots, n$.

The expression $(I-  (AD)^*(AD^2A^*)^{-1}AD) $ seems to be complicated, and also it tacitly assumes the invertibility of $k\times k$ matrix $AD^2A^*$. In fact, this expression is just precisely the orthogonal projection on the subspace $K(x), x\in \Omega,$ where $K(x) = \Ker A D(x)$.

Therefore, we can rewrite \eqref{MMC} as follows:
\begin{equation}
\label{MMCKer}
\begin{cases}
P_{K(x)} (A^* CA\bullet \Hess  B)P_{K(x)}  (x)\le 0\,,\,\, \forall x\in \Omega\\
\det_{n-k} P_{K(x)} (A^* CA\bullet \Hess  B)P_{K(x)} =0\,,\,\, \forall x\in \Omega\,.
\end{cases}
\end{equation}

\begin{remark}
Hence, 1) we should not care too  much about the invertibility of $AD^2A^*$, if it is not invertible, we just understand \eqref{MMC} as   \eqref{MMCKer}; 2) if all entries of $\nabla B(x)$ are non-zero for all $x\in \Omega$ (which is very often the case in applications to isoperimetric inequalities) then $AD^2A^*$ is always invertible (this is just Binet--Cauchy formula and our assumption of full rank of $A$).
\end{remark}

\begin{remark}
Assume $AD$ has full rank. If $k=n$ then our condition \eqref{MMC} becomes trivial and is always true. 
If $k=n-1$ then $P_{\mathrm{ker}}$ has rank $1$ and therefore (\ref{MMC}) holds if and only if 
$$
\mathrm{Tr}(P_{\mathrm{ker} \;AD}\left(A^{*}CA\bullet \mathrm{Hess}\, B\right)P_{\mathrm{ker} \;AD})=
$$
\begin{equation}
\label{knminus1}
\sum_{j} B_{jj}\langle Ca_{j}, a_{j}\rangle - \sum_{i,j}B_{ij}B_{i}B_{j} \langle Ca_{i}, a_{j}\rangle \langle (AD^{2}A^{*})^{-1}a_{i},a_{j}\rangle = 0.
\end{equation} 

 In particular if $B(x_{1},\ldots, x_{n})=x_{n}-H(x_{1},\ldots, x_{n-1})$ where $H$ is a smooth function of $n-1$ variables such that $\partial_{j} H \neq 0$ for all $j=1,\ldots, n-1$, and if 
$a_{1}=(1,0,\ldots, 0)^{T}, \ldots, a_{n-1}=(0,\ldots, 0,1)^{T}$ ($T$ stands for transposition of rows to columns) is a  standard orthonormal basis  in $\mathbb{R}^{k}$, then (\ref{knminus1}) simplifies to 
\begin{align}\label{simetria}
\sum_{i,j=1}^{n-1}\frac{\partial_{ij} H}{\partial_{i} H  \partial_{j} H}a_{ni}a_{nj}c_{ij}= 0.
\end{align}
This is a direct computations, see the next section for details.
\end{remark}

\bigskip

The second type of Bellman PDE (in its form (\ref{simetria})) is ruling such isoperimetric inequalities as Prekopa--Leindler inequality (and thus Brunn--Minkowski inequality) and Ehrhard's inequality (see Section~\ref{Er}).

The reader should be warned that even though \eqref{MC} and \eqref{MMC} look ``almost" the same, they are in fact very different. For example, specially chosen functions $B$ that will prove for us Prekopa--Leindler inequality and Ehrhard's inequality will absolutely not satisfy \eqref{MC} whatever is the choice of $C>0$, but they will satisfy \eqref{MMC} for suitable $C\geq0$.

Below we start with two examples of using modified concavity Bellman PDE \eqref{MC}. Its use will be illustrated by Borell's Gaussian noise stability inequality. We follow closely \cite{Ledoux2014} and \cite{Neeman}. We will also illustrate the use  of \eqref{MC} by ultracontractivity property of Ornstein--Uhlenbeck semigroup.

Then we come to PDEs ruling Prekopa--Leindler and Ehrhard's inequalities. These will be of type \eqref{MMC}.

In Section~\ref{review} briefly describes classical isoperimetric inequalities which we have proved in the current paper. 

\section{Borell's Gaussian noise stability and \eqref{MC} PDE}
\label{BGNS}

Here we follow closely the paper of Ledoux~\cite{Ledoux2014} and our previous paper \cite{IvVo} in what concerns the use of modified concavity PDE \eqref{MC}. We give descriptions in one dimensional case (rank-1 case), and for arbitrary dimension (general rank) we refer the reader to Section~\ref{grank}.

Let $X$ and $Y$ be to standard real Gaussian variable but they are not independent: $\E XY =p, 0<p<1$.

One fixes two numbers $u, v\in [0,1]$ and one looks through all the   sets $A, B$ in $\R$  such that
$$
\gamma_1(A)= u,\,\, \gamma_1(B)=v,
$$
where $\gamma_s$ is a standard Gaussian measure in $\R^s$. One wishes to solve the following isoperimetric problem: maximize (over $A, B$) the probability
$$
\mathcal{P}( X\in A, Y\in B)\,.
$$

First we reformulate the problem in an obvious way, and then we apply \eqref{MC} approach to solve it.

First remark is that we can consider independent standard Gaussians $X, Y$, but now we look at the pair $X, pX+\sqrt{1-p^2} Y$ and we maximize over $A, B \subset \R^1$
$$
\mathcal{P}( X\in A, pX+\sqrt{1-p^2} Y\in B)\,.
$$
It is reasonable to think, and this will be proved, that this supremum--let us call it $b(u,v)$-- coincides with the following supremum
\begin{align*}
&\B^{\sup}(u, v) =\\
& \sup\left\{\int f(x)g(px+\sqrt{1-p^2} y) d\gamma_2(x, y):\, \int fd\gamma_1= u, \int gd\gamma_1 =v, 0\le f\le 1, 0\le g\le 1\right\}\,.
\end{align*}
Of course 
\begin{equation}
\label{simple}
b(u, v)\le \B^{\sup}(u, v)\,.
\end{equation}

Our goal is to show how using \eqref{MC} we can find the formula for $b(u, v)$ and to prove that $b(u,v)=\B^{\sup}(u,v)$.

\begin{theorem}
\label{gMC}
A locally bounded  function $B(u_1,\dots, u_n)$ satisfies inequality of \eqref{MC}  with matrix $A$ of size $k\times n$ with columns $a_1,\dots a_n$ and $C=I_k$ if and only if 
$$
\int B(u_1(a_1\cdot x), \dots, u_n(a_n\cdot x))d\gamma_k(x) \le B\left(\int u_1 (x) d\gamma_1,\dots, \int u_n(x) d\gamma_{1}\right)\,,
$$
for all smooth bounded functions $u_{j}$, 
where $\|a_i\|=1, i=1, \dots, n$.
\end{theorem}

It is very easy to make a change of variables and to have this result for any $C>0$ and any vectors $a_i\neq 0$: 

\begin{cor}
\label{gMCcor}
Function $B(u_1,\dots, u_n)$ satisfies \eqref{MC} (first inequality) with matrix $A$ of size $k\times n$ with columns $a_1,\dots a_n$ and $C>0$ if and only if 
$$
\int B(u_1\langle C^{1/2}a_1,x\rangle, \dots, u_n\langle C^{1/2}a_n,x\rangle)d\gamma_k(x) \le
$$
$$
 B\left(\int u_{1}(x\sqrt{\langle Ca_{1},a_{1}\rangle}) d\gamma_1,\dots, \int u_n(x\sqrt{\langle Ca_{n},a_{n}\rangle}) d\gamma_1\right)\,.
$$
for all smooth bounded functions $u_{j}$. 
\end{cor}

Let us apply Theorem \ref{gMC} to $\vec{a_1}=(1,0)^T, \vec{a_2}=(p, \sqrt{1-p^2})^T$  and any smooth function $B=B(u,v)$,
given on a square $(u,v)\in Q:=[0,1]^2$ such that for matrix $A= \begin{bmatrix} 1, & p\\0, & \sqrt{1-p^2}\end{bmatrix}$ we have
\begin{equation}
\label{Bryant}
A^*A\bullet \Hess B = \begin{bmatrix} B_{uu}, & pB_{uv}\\pB_{uv}, & B_{uu}\end{bmatrix} \le 0
\end{equation}

Then for any such $B$ the theorem claims this inequality :
\begin{equation}
\label{fg1}
\int B(f(x), g(px +\sqrt{1-p^2} y) \, d\gamma_2(x, y) \le B\left(\! \int f(x) \, d\gamma_1(x), \int g(x)\, d\gamma_1(x)\right)\,.
\end{equation}

This is the same as

\begin{equation}
\label{fg2}
\! \! \int\! \! \! B(f(x), g(px +\sqrt{1-p^2} y) \, d\gamma_2(x, y)\! \!  \le \! \! B\left(\! \int\!  f(x) d\gamma_2(x,y), \! \! \int\!  g(px+\sqrt{1-p^2}y)d\gamma_2(x,y)\right)\,.
\end{equation}

Let us consider only $B$ on $Q$, which satisfies \eqref{Bryant} and also satisfies the following boundary conditions
\begin{equation}
\label{bd}
B(0, y)=0,  B(1, y)=y, y\in [0,1], B(x, 0)= 0, B(x,1)= x,  x\in [0,1]\,.
\end{equation}
Now we can choose $f=1_E, g= 1_F$, where $E, F$ are arbitrary, say, closed sets in $\R^1$. 

Then we get from \eqref{fg1}
\begin{equation}
\label{P1}
\mathcal{P}(X\in E, pX+\sqrt{1-p^2} Y \in F)=\gamma_2 ( \{(x,y): x\in E, px+\sqrt{1-p^2} y \in F\}) \le  B(\gamma_1(E), \gamma_1(F))\,,
\end{equation}
or 
\begin{equation}
\label{P2}
b(u,v)=\sup_{E, F\subset \R^1: \gamma_1(E) =u, \gamma_1(F)=v}\mathcal{P}(X\in E, pX+\sqrt{1-p^2} Y \in F) \le  \inf_{B: B\in \eqref{Bryant}, \eqref{bd}}B(u,v)\,.
\end{equation}

Let
$$
\Phi(a):=\frac{1}{\sqrt{2\pi}}\int_{-\infty}^a e^{-x^2/2} dx\,.
$$
Let us choose $E,F$ as rays, $E=(-\infty, a), F= (-\infty, b)$, where $a,b$ are chosen $\gamma_1(E) =u, \gamma_1(F)=v$, that is
\begin{equation}
\label{Phi}
a=\Phi^{-1}(u), b=\Phi^{-1}(v)\,.
\end{equation}
Then
$$
\mathcal{P}(X <a, pX+\sqrt{1-p^2} Y <b) \le \inf_{B: B\in \eqref{Bryant}, \eqref{bd}}B(\Phi(a),\Phi(b))\,.
$$

We want to show the opposite inequality (thus the equality). It has been made clear above that it is enough to check that
the function
$$
\B(u, v) := \mathcal{P}(X <\Phi^{-1}(u), pX+\sqrt{1-p^2} Y <\Phi^{-1}(v)) 
$$
satisfies \eqref{bd}  and also satisfies \eqref{Bryant}. Relation \eqref{bd} is obvious from the definition of $\B$, we are left to verify \eqref{Bryant}.

Moreover, we will see that $\B$ is ``the nest" function satisfying \eqref{Bryant} and boundary conditions \eqref{bd}, in the sense that the following  ``saturation" of  non-positivity of modified Hessian matrix  holds:
\begin{equation}
\label{Bryant_p_det}
\B_{uu}\B_{vv}-p^2 \B_{uv}^2=0\,, \forall  (u,v)\in Q\,.
\end{equation}

\begin{remark}
It is clear that to satisfy inequality \eqref{Bryant} it is sufficient to satisfy equation \eqref{Bryant_p_det}  and inequality $B_{uu}+B_{vv} \le 0$, or even just either $B_{uu}<0$ or $B_{vv}<0$.
\end{remark}

To this end we write $\B(u, v)$ in a different form. We change the variable in the integral:
$$
\frac1{2\pi} \int f(x) g (px +\sqrt{1-p^2} y) e^{-x^2/2} e^{-y^2/2} dxdy = \int f(x) g(z)  K_p(x, z)	e^{-x^2/2} e^{-z^2/2}dxdz
$$
and easily check that
$$
K_p(x, z)= \frac1{2\pi} e^{-\alpha x^2 -\alpha z^2 +\frac{2\alpha}{p} xz},\,\, \text{where}\,\,\alpha= \frac{p^2}{2(1-p^2)}\,.
$$
Plugging into the above formula $f=1_{(-\infty, a)}$, $g=1_{(-\infty, b)}$, $a=\Phi^{-1}(u), b=\Phi^{-1}(v)$, we get
\begin{align}
&2\pi \B(u, v) :=2\pi  \mathcal{P}(X <\Phi^{-1}(u), pX+\sqrt{1-p^2} Y <\Phi^{-1}(v)) =\\  &\!\!\ \int_{-\infty}^{\Phi^{-1}(u)}\!\!\!\int_{-\infty}^{\Phi^{-1}(v)}e^{-\alpha x^2 -\alpha z^2 -\frac{2\alpha}{p} xz}e^{-x^2/2} e^{-y^2/2}dxdz= \int_{-\infty}^{\Phi^{-1}(u)}\int_{-\infty}^{\Phi^{-1}(v)}\!\!\! K_p(x, z) d\gamma_2(x,z) \,.
\end{align}

Direct calculation gives (let $\varphi:=\Phi'$)
\begin{align*}
&\B(u,v)=\mathcal{P}\left( X \leq \Phi^{-1}(u),\; Y \leq  \frac{\Phi^{-1}(v)-pX}{\sqrt{1-p^{2}}}\right)=\\
&\int_{-\infty}^{\Phi^{-1}(u)}\int_{-\infty}^{\frac{\Phi^{-1}(v)-ps}{\sqrt{1-p^{2}}}}\varphi(t)\varphi(s)dtds=\int_{-\infty}^{\Phi^{-1}(v)}\int_{-\infty}^{\frac{\Phi^{-1}(u)-ps}{\sqrt{1-p^{2}}}}\varphi(t)\varphi(s)dtds;\\
&\B_{u}= \int_{-\infty}^{\frac{\Phi^{-1}(v)-p\Phi^{-1}(u)}{\sqrt{1-p^{2}}}}\varphi(t)dt;\\
&\B_{uu}=\varphi\left(\frac{\Phi^{-1}(v)-p\Phi^{-1}(u)}{\sqrt{1-p^{2}}} \right) \frac{-p}{(1-p^{2})^{1/2}\varphi(\Phi^{-1}(u))};\\
&\B_{uv}=\varphi\left(\frac{\Phi^{-1}(v)-p\Phi^{-1}(u)}{\sqrt{1-p^{2}}} \right) \frac{1}{(1-p^{2})^{1/2}\varphi(\Phi^{-1}(v))};\\
&\B_{v}=\int_{-\infty}^{\frac{\Phi^{-1}(u)-p\Phi^{-1}(v)}{\sqrt{1-p^{2}}}}\varphi(t)dt;\\
&\B_{vv}=\varphi\left(\frac{\Phi^{-1}(u)-p\Phi^{-1}(v)}{\sqrt{1-p^{2}}} \right) \frac{-p}{(1-p^{2})^{1/2}\varphi(\Phi^{-1}(v))};
 \end{align*}
 It is clear that $\B_{uu}, \B_{vv}\leq 0$ and 
 \begin{align*}
 \B_{uu}\B_{vv}-p^{2}\B_{uv}^{2}=0.
 \end{align*}

Hence, \eqref{Bryant} (and  also \eqref{Bryant_p_det} are satisfied (so we used the solution of Bellman PDE \eqref{MC} of the first type for our matrix $A=\begin{bmatrix} 1, & p\\0,& \sqrt{1-p^2}\end{bmatrix}$).
To prove 
$$
\inf_{B\in \eqref{Bryant}, \eqref{bd}} B(u, v) = \B(u, v)
$$
(that is the first description of $\B$) we used only boundary condition and inequality \eqref{Bryant}.  Notice that it is  also proved that
$$
\B(u,v) =\B^{\sup}(u,v).
$$
This is the second description of $\B$.

\bigskip

By Theorem~\ref{gMC} any smooth function $B$ satisfying for all $f,g$, $0\le f\le 1, 0\le g\le 1$, 
$$
\int B(f(x), g(px +\sqrt{1-p^2} y) \, d\gamma_2(x, y) \le B\left(\! \int f(x) \, d\gamma_1(x), \int g(x)\, d\gamma_1(x)\right)\,
$$
will also satisfy pointwise inequality \eqref{Bryant}.

\bigskip

This gives the third description of $\B$, it is the saturated  (namely, satisfying $\B_{uu} \B_{vv} -p^2\B_{uv}^2=0$) solution of \eqref{Bryant} with boundary condition \eqref{bd}. In other words, it is a solution of the first type Bellman equation \eqref{MC} with $A= \begin{bmatrix} 1, & p\\0,& \sqrt{1-p^2}\end{bmatrix}$, which satisfies boundary conditions \eqref{bd}.

The fourth description of $\B$ is of course its formula $\B(u,v)=\int_{-\infty}^{\Phi^{-1}(u)}\int_{-\infty}^{\Phi^{-1}(v)}\!\!\! K_p(x, z) d\gamma_2(x,z) =\mathcal{P}(X <\Phi^{-1}(u), pX+\sqrt{1-p^2} Y <\Phi^{-1}(v))$, which we know because this Gaussian extremal problem has been solved beforehand and its solution were known to be rays!

Finally,
we can write
$$
b= \B^{\sup}= \B= \B^{\inf}\,.
$$
Here $\B^{\inf}:=\inf_{B\in \eqref{Bryant}, \eqref{bd}} B(u, v)$. 
It is interesting to ask how one can find other functions $B$ solving \eqref{Bryant} and \eqref{Bryant_p_det} simultaneously. We will show how one can do this in Section \ref{BrSec}.

\subsection{Hypercontractivity of Ornstein--Uhlenbeck semigroup. Young's functions with property \eqref{Bryant_p_det}}
\label{hyper}

Let us consider again functions $B$ that give us
\begin{align*}
 \int_{\mathbb{R}^{2}}B(\varphi(x),\psi(px+\sqrt{1-p^2}\; y))d\gamma_{2} \leq B\left(\int_{\mathbb{R}^{1}}\varphi d\gamma_{1}, \int_{\mathbb{R}^{1}}\psi d\gamma_{1} \right).
 \end{align*}

For that we know it is enough to have \eqref{Bryant_p_det} and $B_{uu}, B_{vv}\le0$. Now let us try to choose $B$ in a very simple form
$$
B(u, v) = u^{1/a}v^{1/b}\,.
$$
It is easy to calulate that \eqref{Bryant_p_det} holds  if and only if $1\leq a, 1\leq b$ and 
\begin{align*}
(a-1)(b-1)-p^2 \geq 0.
\end{align*}
This means that if we denote $\varphi=f^{a}, \psi=g^{b}$ and choose $t$ from the relationship $p:=e^{-t}$, then we have inequality  involving Ornsten--Uhlenbeck semigroup $P_t$:
\begin{align}
\label{hyperineq}
\!\!\!\!\!\!\int_{\mathbb{R}^{n}}f \cdot P_{t} g d\gamma =\int_{\mathbb{R}^{2n}}f(x)g(e^{-t}x+\sqrt{1-e^{-2t}}\; y) d\gamma \leq 
\left(\int_{\mathbb{R}^{n}}f^{a}d\gamma \right)^{1/a}\left(\int_{\mathbb{R}^{n}}g^{b}d\gamma \right)^{1/b}.
\end{align}
We obtain hypercontractivity  for Ornstein--Uhlenbeck semigroup $P_{t}$:
\begin{cor}
\begin{align*}
\|P_{t}g\|_{L^{Q}(d\gamma)} \leq \|g\|_{L^{P}(d\gamma)},
\end{align*}
iff
\begin{align*}
 Q-1 \le e^{2t} (P-1)\,.
\end{align*}
\end{cor}

In fact, 
taking supremum over $f\in L^{a}(d\gamma)$ in \eqref{hyperineq} and setting
$Q=\frac{a}{a-1}$, $P=b$ we get the inequality of the Corollary  under the condition $(a-1)(b-1)-e^{-2t}\geq 0$, which can be rewritten in terms of $P,Q \geq 1$ as $\frac{P-1}{Q-1} -e^{-2t}\ge0$.

\subsection{The proof of Theorem \ref{gMC}}
\label{proof_gMC}

\begin{proof}
Let us consider semigroups $P_t :=e^{\Delta_1 t}$,  $\mathcal{P}_t:=e^{\Delta_k t}$, where $\Delta_k$ is Laplacian in $\R^k$. We already observed the following simple commutation relations: if $a, v$ are vectors in $\R^k$ and $\|a\|=1$ then
$$
(P_t F) (a\cdot v) = (\mathcal{P}_t f)(v),\,\,\text{where}\,\, f(w):= F(a\cdot w), w\in \R^k\,.
$$
The claim of Theorem \ref{gMC} can be then rewritten as follows
$$
(\mathcal{P}_{1/2} B(B(u_1(a_1\cdot w), \dots, u_1(a_1\cdot w))) (0) \le  B((\mathcal{P}_{1/2}u_1(a_1\cdot w))(0), \dots (\mathcal{P}_{1/2}u_n(a_n\cdot w))(0)\,,
$$
or for shortness
just the following inequality with $t=1/2, v=0$:
\begin{equation}
\label{ptpt}
\mathcal{P}_t B (\vec{u})(v) \le  B ((\mathcal{P}_t\vec{u}))(v) )
\end{equation}
Here the vector function $\vec{u}$ has a special form, 
$$
\vec{u}(w):= (u_1(a_1\cdot w), \dots, u_n(a_n\cdot w)), w\in \R^k\,.
$$
We will need also (we assume that $u_1, \dots, u_n$ are smooth and bounded)
$$
\vec{u'}(w):= (u_1'(a_1\cdot w), \dots, u_n'(a_n\cdot w)), w\in \R^k\,.
$$

Notice that if inequality \eqref{ptpt} is satisfied  for particular $t, v$, say $t=1/2$, $v=0\in \R^k$, but for all functions $u_1, \dots, u_n$, then it must be automatically satisfied for all $t>0, v\in \R^k$. Indeed, test the inequality on  the shifts and dilations of $u_{j}$, namely $\tilde{u}_{j}(y)=u_{j}(a_{j}\cdot v + y\sqrt{2t})$, and use (\ref{ax}).

 So just as well we need to prove 
\begin{equation}
\label{BPt}
\mathcal{P}_t B (\vec{u})(x) \le  B ((\mathcal{P}_t\vec{u})(x) )
\end{equation}
 for all positive $t$ and all $x\in \R^k$.

To prove \eqref{ptpt} consider the function in $\R_+^{k+1}$:
$$
V(x, t):= B ((\mathcal{P}_t\vec{u})(x) ) - (\mathcal{P}_t B (\vec{u}))(x), t\ge 0, x\in \R_k
$$
and notice that a direct computation gives us the equality
$$
\big(\Delta_k-\frac{\partial}{\partial t}\big) V(x, t) = \big(\Delta_k-\frac{\partial}{\partial t}\big) B ((\mathcal{P}_t\vec{u}))(x) ) =
$$
\begin{equation}
\label{DeltaBof}
\La A^*A\bullet (\text{Hess} B)(\mathcal{P}_t\vec{u}))(x)) (\mathcal{P}_t\vec{u'}))(x), (\mathcal{P}_t\vec{u'}))(x)\Ra \le 0
\end{equation}
by the first part of our assumption \eqref{MC}.  Also  $V(x,0) =0$ obviously. Then by minimum principle (see, for example \cite{Fjohn}) we  get $V(x, t) \ge 0$ everywhere.

\bigskip

For the converse, we already noticed that inequality in Theorem~\ref{gMC} implies pointwise inequality \eqref{BPt}, that is $V(x, t)\ge 0$.  Now direct computation gives 
$$
0\leq \lim_{t\to 0} \frac{V(x,t)-V(x,0)}{t}=-\La A^*A\bullet ((\text{Hess} B)\vec{u}(x)) \vec{u'}(x), \vec{u'}(x)\Ra.
$$

Since $\vec{u}$ (and hence $\vec{u'}$) is arbitrary  Theorem \ref{gMC} is proved.

\end{proof}

\begin{remark}
Suppose $C$ is a $k\times k$ symmetric matrix and $C>0$. Then we could have consider the semigroup $\mathcal{P}_t^C:= e^{tL_C}$, where $L_C:= \sum_{i, j=1}^k c_{ij} \frac{\partial^{2}}{\partial x_i \partial x_j}$. Given that $A^*CA\bullet \text{Hess} B\le 0$ everywhere in $\Omega$ we would obtain that
the following analog of \eqref{BPt} also holds
\begin{equation}
\label{ptptC}
\mathcal{P}_t^C B (\vec{u})(x) \le  B ((\mathcal{P}_t^C\vec{u})(x) ),\,\, \forall t\ge 0, \forall x\in \R^k\,.
\end{equation}

It is interesting to remark that if we {\it do not assume} $C>0$ (or even $C\ge 0$),  and we only assume that $\langle Ca_{j},a_{j}\rangle >0$ for all $j$, then  certain shadow of this pointwise inequality still holds. It will be an integral inequality. Notice first that equality
$$
\big(L_C-\frac{\partial}{\partial t}\big) B (u_1(a_1\cdot x,t), \dots , u_n(a_n\cdot x,t) ) =
$$
\begin{equation}
\label{LCBof}
\La A^*CA\bullet (\Hess B)(\vec{u}(x,t)) \vec{u'}(x,t), \vec{u'}(x,t)\Ra
\end{equation}
of course does not require any positivity of $C$. It follows from \eqref{LC}: here each flow $u_{j}(y,t)$ is with different speed, namely $\langle Ca_{j},a_{j}\rangle\frac{\partial^{2}}{\partial y^{2}}u_{j}(y,t)=\frac{\partial}{\partial t}u_{j}(y,t)$.  Then we integrate this equality over $\R^k$. Then
\begin{equation}
\label{integral}
\frac{d}{d t}\int_{\R^k} B ( u_1(a_1\cdot x,t), \dots , u_n(a_n\cdot x,t) )  dx \ge 0\,.
\end{equation}
Here one used $\lim_{R\to\infty}\int_{x\in \R^k,: |x|\le R} L_C B (u_1(a_1\cdot x,t), \dots , u_n(a_n\cdot x,t) )  dx =0$, which can be seen by Stokes theorem under some mild assumptions on $B$ (see~\cite{IvVo}).  It is just an integration by parts and the fact that $u(y,t)$ goes to zero fast if $y$ goes to infinity and $u$ is a function with compact support. In particular, if we denote by $\E(t)$ the following ``energy" functional
$$
\E(t) := \int_{\R^k} B ( u_1(a_1\cdot x,t), \dots , u_n(a_n\cdot x,t) )dx \,,
$$
we obtain the integral inequality
\begin{equation}
\label{integralE}
\E(t)\ge \E(0) \,\, \forall t\ge 0\,.
\end{equation}

Of course this inequality immediately follows from the much stronger pointwise inequality  \eqref{ptptC} if $C>0$.  To see this just integrate \eqref{ptptC} with respect to Lebesgue measure $dx$ in $\R^k$.
\end{remark}

\begin{remark}
An interesting (and sometimes useful) observation is that we can think that $P_t, \mathcal{P}_t, \mathcal{P}_t^C$ are semigroups of Ornsein--Uhlenbeck type. We can think that all second order differential operators we used above have a drift (a first order part). Absolutely nothing changes and \eqref{ptptC} holds. The integral inequalities \eqref{integral}, \eqref{integralE} will also hold with one small change: the integration should be with respect to the Gaussian measure $d\gamma_k(x)$. Here is a Gaussian analog of \eqref{integralE}:
\begin{equation}
\label{integralEg}
\E_g(t)\ge \E_g(0) \,\, \forall t\ge 0\,,
\end{equation}
where 
$$
\E_g(t) := \int_{\R^k} B ((P_t u_1))(a_1\cdot x), \dots , (P_t u_n))(a_n\cdot x) )d\gamma_k(x) \,.
$$
Of course here $P_t$ is an Ornstein--Uhlenbeck semigroup. Actually it is now a great advantage. We want to make $t\to\infty$ in \eqref{integralE} and/or \eqref{integralEg}. It is not so easy to do that in  \eqref{integralE}  (but one can do this sometimes, see \cite{IvVo}), but in \eqref{integralEg} it is very easy to pass to the limit because measure $d\gamma_k$ is finite and because one has a uniform convergence of $(P_t u)(y)$ to $\int u d\gamma_1$ for Ornstein--Uhlenbeck semigroup $P_t$. Coming to the limit  $t\to\infty$ in \eqref{integralEg} we immediately obtain $E_g(\infty)\ge E_g(0)$ or
\begin{equation}
\label{integralEgg}
B\left(\int u_1 d\gamma_1,\dots, \int u_n d\gamma_1\right) \ge \int_{\R^k}B (u_1(a_1\cdot x), \dots , u_n(a_n\cdot x) )d\gamma_k(x)\,,
\end{equation}
which gives us another  proof of Theorem \ref{gMC}.

\end{remark}

\subsection{Solving $B_{uu}B_{vv} -p^2B_{uv}^2=0$.}
\label{BrSec}

The following question was asked in \cite{Ledoux2014} and \cite{IvVo}.
\begin{problem}
Describe all possible solutions of the partial differential system of equality and inequality
\begin{align*}
A^{*}A \bullet \mathrm{Hess}\, B \leq 0 \quad \text{and} \quad \det(A^{*}A\bullet \mathrm{Hess}\, B)=0.
\end{align*}
\end{problem} 

Let us consider the following particular. Let $B\in C^{2}$ be given in some rectangular domain.
 Let $n=k=2$ and take $A=(a_{1}, a_{2})$ where $a_{1}, a_{2}\neq 0$. Then we must have 
\begin{align*}
A^{*}A \bullet \mathrm{Hess}\, B = \left( {\begin{array}{cc}
             |a_{1}|^{2} B_{11} & a_{1}\cdot a_{2} \,B_{12}   \\
             a_{1}\cdot a_{2} \,B_{12}  & |a_{2}|^{2}B_{22}             
                \end{array} } \right) \leq 0\quad \text{and} \quad \det(A^{*}A\bullet \mathrm{Hess}\, B)=0.
\end{align*} 
 If $a_{1}\cdot a_{2}=0$ then $B$ has to be separate concave functions such that $B_{11}B_{22}=0$ and these are the all possible solutions. Therefore we assume that $a_{1}\cdot a_{2}\neq 0$. Then we see that $B$ must be separate concave function and moreover 
 \begin{align*}
 \frac{|a_{1}|^{2}|a_{2}|^{2}}{|a_{1}\cdot a_{2}|^{2}} B_{11}B_{22}- B_{12}^{2}=0.
 \end{align*}
 So in the case $n=k=2$ the problem reduces to the following one 
 \begin{problem}
 Let $|c| \in [1, \infty)$ and let $B\in C^{2}$ be given on some rectangular domain in $\mathbb{R}^{2}$. Characterize all possible separately concave functions $B$ such that 
 \begin{align}\label{modmong}
c^{2}   \,B_{11}B_{22}-B_{12}^{2}=0.
 \end{align}  
 \end{problem}

\bigskip

 The  case $|c|=1$ corresponds to the homogeneous Monge--Amp\`ere equation and, thus, to  developable surface and the characterization of these surfaces are mostly known. The possible references are Pogorelov~\cite{Po}, Vasyunin--Volberg~\cite{VaVo}, Ivanisvili et al \cite{iosvz1,iosvz2,iosvz3,iosvz4,iosvz5,iosvz6}.

\bigskip

For general $|c| >1$ we can give local characterization. Namely, we will show that the above equation can be reduced to the following one 
 \begin{align*}
 \frac{\partial f}{\partial \bar{z}}=\bar{f}
 \end{align*}
 for some appropriate $f$ (see below). 
 
 For separately concave $B(x,y)$ set $B_{xx}=-p^{2}, B_{yy}=-q^{2}$. Then equation (\ref{modmong}) implies that $B_{xy}=cpq$. We also have 
 \begin{align}
 &-2pp_{y}=cqp_{x}+cpq_{x}, \label{mo11}\\
 &-2qq_{x} = cqp_{y}+cpq_{y}.\label{mo12}
 \end{align}
 Further we assume that $p, q \neq 0$. Assume that locally the map $p,q : (x,y) \to \mathbb{R}^{2}$ is invertible, and let $(x,y)$ be its inverse map. Then 
 \begin{align*}
 \left( {\begin{array}{cc}
             p_{x} & p_{y}   \\
             q_{x}  & q_{y}             
                \end{array} } \right) = \left( {\begin{array}{cc}
             x_{p} & x_{q}   \\
             y_{p}  & y_{q}             
                \end{array} } \right)^{-1}=
                 \frac{1}{\det(\mathrm{Jacob} (x,y))}\cdot \left( {\begin{array}{cc}
             y_{q} & -x_{q}   \\
             -y_{p}  & x_{p}             
                \end{array} } \right).
 \end{align*}
 Therefore equations (\ref{mo11}) and (\ref{mo12}) take the following form 
 \begin{align*}
 &2p x_{q}=cqy_{q}-cpy_{p},\\
 &2qy_{p}=-cqx_{q}+cpx_{p}.
 \end{align*}
 This can be written as follows 
 \begin{align*}
 &2(px)_{q}=c(qy)_{q}-c(py)_{p},\\
 &2(qy)_{p}=-c(qx)_{q}+c(px)_{p}.
 \end{align*}
We set $\tilde{U}(p,q)=p x(p,q)$ and $\tilde{V}(p,q)=q y(p,q)$. Then we obtain 
\begin{align*}
2\tilde{U}_{q}=c\tilde{V}_{q}-c \left(\frac{\tilde{V} p}{q}\right)_{p},\\
2\tilde{V}_{p}=-c \left(\frac{\tilde{U} q}{p}\right)_{q} + c\tilde{U}_{p}.
\end{align*}
After the logarithmic substitution $\tilde U(p,q)=M(\ln p, \ln q)$ and $\tilde V(p,q)=N(\ln p, \ln q)$ we obtain the linear equation
\begin{align*}
&2 M_{2}=c(N_{2}-N-N_{1}),\\
&2N_{1}=c(-M-M_{2}+M_{1}).
\end{align*}

By setting $k=2/c \in (-2,2)$, this can be rewritten as follows 
\begin{align*}
 \left( {\begin{array}{c}
           N   \\
             M            
                \end{array} } \right)= \left( {\begin{array}{cc}
             -1 & 1   \\
             -k  & 0             
                \end{array} } \right)
                \left( {\begin{array}{c}
             N_{1}   \\
             N_{2}            
                \end{array} } \right)+
               \left( {\begin{array}{cc}
            0 & -k   \\
            1  & -1             
                \end{array} } \right)
                \left( {\begin{array}{c}
             M_{1}   \\
             M_{2}            
                \end{array} } \right).\\              
\end{align*}
We need the following technical lemma.

\begin{lemma}
If the vector function $\vec{N}(x,y) =(N,M): \Omega \subset \mathbb{R}^{2} \to \mathbb{R}^{2}$ satisfies the following  first order system of linear differential equations 
%\begin{align*}
%\vec{N}=P\vec{N}_{1}+Q\vec{N}_{2}
%\end{align*}

\begin{align*}
 \left( {\begin{array}{c}
           N   \\
             M            
                \end{array} } \right)=
P
                \left( {\begin{array}{c}
             N_{1}   \\
             N_{2}            
                \end{array} } \right)+
              Q
                \left( {\begin{array}{c}
             M_{1}   \\
             M_{2}            
                \end{array} } \right).\\  
\end{align*}
for some  invertible $2\times 2$ matrices $P,Q$ where
\begin{align}
\label{QP}
QP^{-1}= \left( {\begin{array}{cc}
         -2t & \delta^{2}       \\                        
            -1 & 0    \end{array} } \right)
\end{align} 
for some $t \in (-\delta,\delta), \delta >0$ then  after making change of variables 
$\vec{N}(\vec{x})=B\vec{U}(A\vec{x})$,  where
$$
\vec{U}= (U, V)\,,
$$
\begin{align*}
B = \left( {\begin{array}{cc}
         t & \sqrt{\delta^{2}-t^{2}}       \\                        
            1 & 0    \end{array} } \right)\quad \text{and} \,,
\end{align*}
\begin{equation}
\label{chinside}
 A^{T}=\frac{1}{2}P^{-1}\left( {\begin{array}{cc}
         -1 & -\frac{t}{\sqrt{\delta^{2}-t^{2}}}      \\                        
            0 & -\frac{1}{\sqrt{\delta^{2}-t^{2}}}    \end{array} } \right),
\end{equation}
we obtain 
\begin{align*}
\frac{\partial f}{\partial \bar{z}}  =\bar{f},
\end{align*} 
where $f = U+iV$. 
\end{lemma}

\begin{proof}
Set $P=(P_{1},P_{2}),Q=(Q_{1},Q_{2})$  where $P_{i},Q_{j}$ are columns. 

\begin{align*}
 \left( {\begin{array}{c}
           N   \\
             M            
                \end{array} } \right)=N_{1}P_{1}+N_{2}P_{2}+M_{1}Q_{1}+M_{2}Q_{2}.  
                \end{align*}

Now let $N(x,y)=\tilde{N}(\alpha_{1} x +\alpha_{2} y, \beta_{1} x + \beta_{2} y)$ then 

\begin{align*}
&N_{1}=\alpha_{1} \tilde{N}_{1}+\beta_{1} \tilde{N}_{2};\\
&N_{2}=\alpha_{2} \tilde{N}_{1}+\beta_{2} \tilde{N}_{2};\\
&M_{1}=\alpha_{1} \tilde{M}_{1}+\beta_{1} \tilde{M}_{2};\\
&M_{2}=\alpha_{2} \tilde{M}_{1}+\beta_{2} \tilde{M}_{2}.
\end{align*}

So we obtain 
\begin{align*}
 \left( {\begin{array}{c}
           \tilde{N}   \\
             \tilde{M}            
                \end{array} } \right) = &(P_{1}\alpha_{1}+P_{2}\alpha_{2})\tilde{N}_{1}+(P_{1}\beta_{1}+P_{2}\beta_{2})\tilde{N}_{2}+\\
                &(Q_{1}\alpha_{1}+Q_{2}\alpha_{2})\tilde{M}_{1}+(Q_{1}\beta_{1}+Q_{2}\beta_{2})\tilde{M}_{2}.
\end{align*}
Finally we set 
\begin{align*}
\tilde{N}=a_{1}U+b_{1}V\\
\tilde{M} = a_{2}U+b_{2}V.
\end{align*}
and 
\begin{align*}
B=
\left( {\begin{array}{cc}
           a_{1} & b_{1}   \\
            a_{2} & b_{2}            
                \end{array} } \right).
                \end{align*}
Thus we obtain 
\begin{align*}
& \left( {\begin{array}{c}
           U   \\
             V            
                \end{array} } \right)=\\
&B^{-1}[a_{1}(P_{1}\alpha_{1}+P_{2}\alpha_{2})+a_{2}(Q_{1}\alpha_{1}+Q_{2}\alpha_{2})]U_{1}+\\
&B^{-1}[a_{1}(P_{1}\beta_{1}+P_{2}\beta_{2})+a_{2}(Q_{1}\beta_{1}+Q_{2}\beta_{2})]U_{2}+\\
&B^{-1}[b_{1}(P_{1}\alpha_{1}+P_{2}\alpha_{2})+b_{2}(Q_{1}\alpha_{1}+Q_{2}\alpha_{2})]V_{1}+\\
&B^{-1}[b_{1}(P_{1}\beta_{1}+P_{2}\beta_{2})+b_{2}(Q_{1}\beta_{1}+Q_{2}\beta_{2})]V_{2}.
\end{align*}

And we would like to see that 
\begin{align*}
 \left( {\begin{array}{c}
           U   \\
             V            
                \end{array} } \right)= \frac{1}{2}\left( {\begin{array}{cc}
             1 & 0   \\
             0  & -1             
                \end{array} } \right)
                \left( {\begin{array}{c}
             U_{1}   \\
             U_{2}            
                \end{array} } \right)+
               \frac{1}{2}\left( {\begin{array}{cc}
            0 & -1   \\
            -1  & 0             
                \end{array} } \right)
                \left( {\begin{array}{c}
             V_{1}   \\
             V_{2}            
                \end{array} } \right).\\              
\end{align*}
This can hold if and only if 
\begin{align*}
&(P\alpha, Q\alpha)=\frac{1}{2}B\left( {\begin{array}{cc}
             1 &0  \\
             0  &-1         
                \end{array} } \right)B^{-1}=\frac{1}{2}BI^{+}B^{-1};\\
&(P\beta, Q\beta)=   \frac{1}{2} B\left( {\begin{array}{cc}
             0 & -1  \\
             -1  &0         
                \end{array} } \right)B^{-1}=\frac{1}{2}BI^{-}B^{-1},             
\end{align*}
where 
\begin{align*}
\alpha=\left( {\begin{array}{c}
             \alpha_{1}  \\
             \alpha_{2}        
                \end{array} } \right); \quad \beta=\left( {\begin{array}{c}
             \beta_{1}  \\
             \beta_{2}        
                \end{array} } \right).            
\end{align*}
Let $e_{1}=(1,0), e_{2}=(0,1)$, and let us introduce the matrices
\begin{align*}
&B_{1}=\frac{1}{2}\left(BI^{+}B^{-1}e_{1}, BI^{-}B^{-1}e_{1} \right),\\
&B_{2}=\frac{1}{2}\left(BI^{+}B^{-1}e_{2}, BI^{-}B^{-1}e_{2} \right).
\end{align*}
Then the above conditions hold iff 
\begin{equation}
\label{systCORR}
P \cdot (\alpha, \beta)=(P\alpha, P\beta)=B_{1};\,\,
Q \cdot (\alpha, \beta)=(Q\alpha, Q\beta)=B_{2}.
\end{equation}

The system \eqref{systCORR} is overdetermined, it has two equations on one matrix $(\alpha, \beta)$. There is one compatibility condition: 
 $QP^{-1}=B_{2}B_{1}^{-1}$. 

It is easy to calculate $B_1, B_2$ for matrix $B$, which was given in the assumption of the lemma.  Then we can calculate $B_2 B_1^{-1}$ and automatically obtain that it is equal to 
\begin{align*}
 \left( {\begin{array}{cc}
         -2t & \delta^{2}       \\                        
            -1 & 0    \end{array} } \right)\,.
\end{align*} 
(Note that if $r$ and $s$ are corresponding rows of the matrix $B$  then 
\begin{align*}
B_{2}B_{1}^{-1}=\left( {\begin{array}{cc}
             -2\frac{r\cdot s}{|s|^{2}} & \frac{|r|^{2}}{|s|^{2}}  \\
             -1 &0        
                \end{array} } \right),    
\end{align*}
so the claim follows.)
This is precisely the form of $QP^{-1}$ from \eqref{QP}.
This means that the system of equations on matrix $(\alpha, \beta)$ is compatible, and so matrix $(\alpha, \beta)$ is well defined.

Set $(\alpha, \beta):= P^{-1}B_1$.  This is precisely the formula \eqref{chinside}. The lemma is proved. 
\end{proof}

\bigskip

In our case of $P, Q$'s, we have  $P=\left( {\begin{array}{cc}
             -1 & 1   \\
             -k  & 0             
                \end{array} } \right)$, $Q=
               \left( {\begin{array}{cc}
            0 & -k   \\
            1  & -1             
                \end{array} } \right)$, and
 
\begin{align*}
QP^{-1}=\left( {\begin{array}{cc}
             -k & 1  \\
             -1 &0        
                \end{array} } \right)=\left( {\begin{array}{cc}
             -\frac{2}{c} & 1  \\
             -1 &0        
                \end{array} } \right),
\end{align*}
therefore we can apply the lemma and we see that taking  $t =1/c \in (-1,1)$ and $\delta=1$  we have 
\begin{align*}
B=\left( {\begin{array}{cc}
             t & \sqrt{1-t^{2}}  \\
             1 &0        
                \end{array} } \right) \quad     A=      \left( {\begin{array}{cc}
             0 & -\frac{1}{2}  \\
              \frac{1}{4t\sqrt{1-t^{2}}}&-\frac{1}{4} \frac{2t^{2}-1}{t\sqrt{1-t^{2}}}       
                \end{array} } \right).
\end{align*}

This means that if we set 
\begin{align*}
\left( {\begin{array}{c}
             N(x,y)  \\
             M(x,y)        
                \end{array} } \right)=\left( {\begin{array}{cc}
             t & \sqrt{1-t^{2}}  \\
             1 &0        
                \end{array} } \right)\left( {\begin{array}{c}
             U\left(-\frac{y}{2},\frac{x-y(2t^{2}-1)}{4t\sqrt{1-t^{2}}}\right)  \\
             V\left(-\frac{y}{2},\frac{x-y(2t^{2}-1)}{4t\sqrt{1-t^{2}}}\right)     
                \end{array} } \right),
\end{align*}
where by setting $z=x+iy$ for the  function $f(z,\bar{z})=U(x,y) +iV(x,y)$  we have
 \begin{align*}
 \frac{\partial f}{\partial \bar{z}}  = \bar{f}.
 \end{align*}
It is known that all $C^{1}$ solutions of the above equation are real analytic  and they can be represented in terms of power series 
\begin{align*}
f(z) = \sum_{k=0}^{\infty}c_{k} J^{(k)}(z\bar{z})z^{k}+\bar{c}_{k}J^{(k+1)}(z\bar{z})\bar{z}^{k+1},
\end{align*}
where $J(r)$ is modified Bessel $I$-functions whose series representation is 
\begin{align*}
J(r)=\sum_{j=0}^{\infty}\frac{r^{j}}{(j!)^{2}}.
\end{align*}
 
%We will try to trace back the solution when $c_{k}=0$ for all $k\geq 1$ and $c_{0}=1$. In this case $f(z)=J(|z|^{2})+J'(|z|^{2})\bar{z}$
%Hence 
%\begin{align*}
%&U(x,y)=J(x^{2}+y^{2})+J'(x^{2}+y^{2})x;\\
%&V(x,y)=-J'(x^{2}+y^{2})y.
%\end{align*}

%We need to solve for $p=p(x,y)$,  $q=q(x,y)$ and $\vec{p}=(p,q)$ the following equation 
%\begin{align*}
%&U(A\vec{p})=e^{p}x\\
%&tU(A\vec{p})+\sqrt{1-t^{2}}\, V(A\vec{p})=e^{q}y
%\end{align*}
%Then 
%\begin{align*}
%B(x,y)=-\int_{0}^{x}\int_{0}^{u}p^{2}(s,y)dsdu
%\end{align*}
%gives the desired solution. 
%Consider the simple case when $t=\frac{1}{c}=\sqrt{\frac{1}{2}}$. Then $A\vec{p}=(-\frac{q}{2}, \frac{p}{2})$. Hence 
%\begin{align*}
%&J\left(\frac{p^{2}+q^{2}}{4}\right)-\frac{q}{2}J'\left(\frac{p^{2}+q^{2}}{4}\right)=e^{p}x\\
%&J\left(\frac{p^{2}+q^{2}}{4}\right)-\frac{q}{2}J'\left(\frac{p^{2}+q^{2}}{4}\right)-\frac{p}{2}J'\left(\frac{p^{2}+q^{2}}{4}\right)=\sqrt{2} e^{q}y.
%\end{align*}
%which can be simplified to 
%\begin{align*}
%J'\left(\frac{p^{2}+q^{2}}{4}\right)
%\end{align*}

\section{Bellman equation of the second type: PDE that rules the Prekopa--Leindler inequality and Ehrhard inequality}
\label{PrLe}

In the previous section we used the following  minimum principle.  If a smooth function $V(x, t), x\in \R^k, t\ge 0,$ satisfies the growth condition  $V(x,t)\geq -Me^{\lambda |x|^{2}}$ for some nonnegative constants $M, \lambda \geq 0$,  and it is a superharmonic function in this sense
\begin{equation}
\label{caloric}
(L_C -\frac{\partial}{\partial t}) V(x, t) =\left(\sum_{i, j=1}^k c_{ij} \frac{\partial^2}{\partial x_i \partial x_j}-\frac{\partial}{\partial t}\right) V(x,t) \le 0,\, \forall x\in \R^k, t>0
\end{equation}
 $C>0$ then it has minimum principle: $V(x, t)\ge \inf V(\cdot,0)$.

Remember that our $V$ was of the following type 
$$
V(x, t) := B((\mathcal{P}_t^C u_1(a_1\cdot .))(x),\dots, 
(\mathcal{P}_t^C u_n(a_n\cdot .))(x))-\mathcal{P}_{t}(B(\vec{u}))(x)\,.
$$

The requirement \eqref{caloric} transforms into $A^*CA\bullet \text{Hess} B\le 0$. For given $A$ there could be very limited amount of functions $B$ for which there exists a positive $C$ with this property. In fact, in \cite{IvVo} we proved that sometimes one can enumerate all such functions by the  list of Young's functions. These functions provide us with  Brascamp--Lieb inequality inequality (see~\cite{brascamp--lieb, barthe, BCCT, BCCT2, Barthe3, Barthe4})

\bigskip

 However, to have the minimum principle one does not need $V$ to satisfy \eqref{caloric} for all $x, t$. It is easy to see that  it is sufficient to satisfy \eqref{caloric} {\it only} at the points, where $V(x,t)$ ($t>0$ is fixed) has a local minimum in $x$. In particular, it is enough to have $C\geq 0$ such that for any $(x_0, t_0)$, $t_0>0$,
\begin{equation}
\label{hill}
\nabla_{x} V (x_0, t_0) =0, \, \text{Hess}_x V(x_0, t_0) \ge 0\Rightarrow (L_C -\frac{\partial}{\partial t}) V(x_0, t_0) \le 0\,.
\end{equation}
This property (we call it {\it hill property}) was used by Barthe (see~\cite{barthe1}) to give a proof of Ehrhard's inequality. Here we will show how the hill property proves such classical inequalities as Prekopa--Leindler and Ehrhard's inequalities and also gives a whole plethora of isoperimetric inequalities ruled by certain PDE. This PDE will be \eqref{MMC}. 

The reason why the hill property works so well is that checking it allows to have a much bigger supply of functions $B$ such that 
$$
V(x, t) := B((\mathcal{P}_t^C u_1(a_1\cdot .))(x),\dots, 
(\mathcal{P}_t^C u_n(a_n\cdot .))(x))
$$ 
satisfies the hill property. It turns out that there exists a simple and often easily checkable for concrete functions $B$ PDE \eqref{MMC}, which is {\it sufficient} for \eqref{hill} if  
$$
V(x, t) := B((\mathcal{P}_t^C u_1(a_1\cdot .))(x),\dots, 
(\mathcal{P}_t^C u_n(a_n\cdot .))(x))\,.
$$
Let us prove this statement

In what follows we will need the following conditions at infinity:
\begin{equation}
\label{infinity0}
\forall T>0,\, \liminf_{|x|\to\infty} \inf_{t\in [0, T]} V(x, t) \ge 0\,.
\end{equation}

\begin{theorem}
\label{MMC_hill}The following statements hold:
\begin{itemize}
\item[(i)]
If $V(x,t)$ has the hill property (\ref{hill}) with certain $C\geq 0$ and also property (\ref{infinity0}) at infinity then $V(x,0)\geq 0$ for all $x$ implies $V(x,t)\geq 0$ for all $x$. 
\item[(ii)]
Let  $C\geq 0$ be  $k\times k$ matrix such that $\langle Ca_{j}, a_{j}\rangle >0$ for all $j$, and  the first line of \eqref{MMC} is satisfied for $B$. Then  $V(x, t) := B((\mathcal{P}_t^C u_1(a_1\cdot .))(x),\dots, 
(\mathcal{P}_t^C u_n(a_n\cdot .))(x))$ has the hill property. 
Moreover, if in addition $V$ has property (\ref{infinity0}) at infinity then $V(x,0)\geq 0$ for all $x$ implies $V(x,t)\geq 0$ for all $x$.  In particular,  for $x=0$ and $t=1/2$ we have 
\begin{align}\label{Biso}
 B\left(\int_{\mathbb{R}} u_{1}(y\sqrt{\langle Ca_{1},a_{1}\rangle})d\gamma_{1}(y),\ldots, \int_{\mathbb{R}} u_{1}(y\sqrt{\langle Ca_{n},a_{n}\rangle})d\gamma_{1}(y) \right)\geq 0.
\end{align}
\end{itemize}
\end{theorem}

\begin{proof} \quad\\

\textup{(i)} We check that the condition at infinity and the hill property imply the minimum principle: $V(x, t)\ge 0$ for all $x$. Here we follow the proof of Barthe \cite{barthe}. It is enough to show that for any $\eps>0$ we have $V_\eps(x,t):= V(x, t) +\eps t \geq 0$.

First we check that  for any $T$, $V_\eps(x,t)$ does not attain local minimum  in $\mathbb{R}^{k}\times (0,T]$. Indeed, if it does attain a local minimum at point $(x_{0},t_{0})$ then $\Hess_{x} V_{\eps}(x_{0},t_{0})=\Hess_{x} V(x_{0},t_{0})\geq 0, \nabla_{x} V_{\eps}(x_{0},t_{0})=\nabla_{x} V(x_{0},t_{0})=0$ and $\partial_{t} V_{\eps}(x_{0},t_{0})=\partial_{t}V(x_{0},t_{0})+\eps =0$. The last property implies that $\partial_{t} V(x_{0},t_{0})=-\eps <0$. However, the hill property implies that ($L_{C}-\partial_{t})V(x_{0},t_{0})\leq 0$. But notice that $L_{C}V(x_{0},t_{0})=\tr(C\Hess_{x}V)(x_{0},t_{0})\geq 0$ since $C\geq 0$ and $\Hess_{x}V(x_{0},t_{0})\geq 0$. Putting things together we obtain $\partial_{t} V(x_{0},t_{0})\geq 0$, and this contradicts to the fact that $\partial_{t} V(x_{0},t_{0})=-\eps$.

Now suppose $V(x_1, t_1) =-\delta, \delta>0$. Then for very small $\eps$, $V_\eps(x_1, t_1)<0$. Taking into account that $V(x, 0)\ge 0$ and assumption \eqref{infinity0} we conclude that $V_\eps$ must have a local minimum in $\R^k\times (0, t_1]$. This is a contradiction.\\

\textup{(ii)} Let \eqref{MMC} be satisfied (just its first line). Let $D$ denotes $n\times n$ diagonal matrix such that it has $\nabla B$ on the diagonal, and let $I$ denotes identity matrix.  
Let  
$$
V(x, t) := B((\mathcal{P}_t^C u_1(a_1\cdot .))(x),\dots, 
(\mathcal{P}_t^C u_n(a_n\cdot .))(x))\,.
$$ 
Let us rewrite $V(x, t)$ in the following form.
\begin{align}
\label{bba}
V(x,t)=B(u_{1}(a_{1}\cdot x,t),\ldots, u_{n}(a_{n}\cdot x,t))\,,
\end{align}
where $u_{j}(z,t)$ denotes $\frac{\partial}{\partial t}u_{j}(y,t)=\langle Ca_{j},a_{j}\rangle u''_{j}(y,t)$.

Note that   $\nabla_{x} V=0$ implies that $\sum_{p} \frac{\partial B}{\partial u_{p}}u'_{p}a_{pj}=0$ for all $j=1,\ldots, k$. This means that $ADu'=0$ where $u'=(u'_{1},\ldots, u'_{n})$. Condition $ADu'=0$ implies that $u'=P_{\mathrm{ker}\; AD}\,\s{y}$
 for some $\s{y} \in \mathbb{R}^{n}$.  After this  the direct computations show
\begin{align*}
&Tr(C\; \mathrm{Hess}_{x}\; V(x,t))-\frac{\partial V}{\partial t}=\left(\sum_{i,j}c_{ij}\frac{\partial^{2}}{\partial x_{i} \partial x_{j}}- \frac{\partial }{\partial t}\right)V = \langle A^{*}CA\bullet \mathrm{Hess}\, B\, u', u'\rangle =\\
&\langle P_{\mathrm{ker}\; AD}\left(A^{*}CA\bullet \mathrm{Hess}\, B\right)P_{\mathrm{ker}\; AD}\,\s{y},\s{y}\rangle \leq 0.
\end{align*}
The second line is precisely the first part of \eqref{MMC}, which we assumed in the theorem.
%and this finishes the proof because $Tr(C\; \mathrm{Hess}_{x}\; V(x,t))\geq 0$. 

\bigskip

 %In fact, suppose not. Then for some $t_0>0$ we have negative values of $V$. But at infinity this function is zero. So we have some global minimum  of $V(x, t_0)$ on $\R^k \times[0, t_0]$, and it is strictly negative. Consider $V_\eps(x,t):= V(x, t) +\eps t$ with very small positive $\eps$. Then it is clear that  $V\eps(x,t)$ also has some global minimum  of $V(x, t)$ on $\R^k \times[0, t_0]$, and it is strictly negative. Let it be at point $(x_0, t_0')$ ($t_0'$ can be equal to  $t_0$). Then, of course, $\nabla V_{x} (x_0, t_0') =0, \, \text{Hess}_x V_\eps(x_0, t_0') \ge 0$ because it is a minimum point in variable $x$ for function $V$ (because it s a global minimum of $V_\eps$). Then the hill property says that $(L_C -\frac{\partial}{\partial t}) V_(x_0, t_0') \le 0$. But notice that $L_C V_\eps(x_0, t_0') = \tr (C\text{Hess}V)(x_0, t_0')$ and this quantity is positive because $C>0$ and $\text{Hess}_x V(x_0, t_0') \ge 0$. Putting things together we obtain $\frac{\partial}{\partial t} V(x_0, t_0')\ge 0$. Then $\frac{\partial}{\partial t} V_\eps(x_0, t_0')> 0$, which is impossible because it was a global minimum of this function.

\end{proof}

For $V(x, t) := B( u_1(a_1\cdot x,t),\dots, u_n(a_n\cdot x,t))$, let us see that \eqref{infinity0}  is practically implied by the requirement that $V(x, 0)\ge 0$ for all $x\in \R^k$.  

\bigskip

We would like to show that  for any $T>0$ we have $\liminf_{|x| \to \infty} \inf_{0\leq t \leq T} B((P_{t}^{C} \vec{u})(x))\geq 0$: 
\begin{align*}
\liminf_{|x| \to \infty} \inf_{0\leq t \leq T} B\left(\ldots, \frac{1}{\sqrt{4 \pi \langle Ca_{j},a_{j}\rangle t}}\int_{\mathbb{R}}u_{j}(y)e^{-\frac{(\langle a_{j},x\rangle-y)^{2}}{4t\langle Ca_{j},a_{j}\rangle}}dy,\ldots \right)\geq 0.
\end{align*}

For a bounded function $u_j$ with compact support
$$
\lim_{|a_{j}\cdot x|\to\infty}\sup_{t\in [0,T]}\big|\frac{1}{\sqrt{4 \pi \langle Ca_{j},a_{j}\rangle t}}\int_{\mathbb{R}}u_{j}(y)e^{-\frac{(\langle a_{j},x\rangle-y)^{2}}{4t\langle Ca_{j},a_{j}\rangle}}dy\big| =0.
$$

Then $0\le V(x, 0)=B((u_1)(a_1\cdot x),\dots, 
(u_n)(a_n\cdot x))\to B(0)$ if $x\to\infty$ in such a way that $\min_{i\in [1,\dots, k]}|a_i\cdot x| \to \infty$.
Then
$$
\lim_{\min_{i\in [1,\dots, k]}|a_i\cdot x| \to \infty}V(x, t) =\lim_{\min_{i\in [1,\dots, k]}|a_i\cdot x| \to \infty} B((P_t^C u_1)(a_1\cdot x),\dots, 
(P_t^C u_n)(a_n\cdot x)) =B(0) \ge 0\,.
$$

However, we need \eqref{infinity0}, which is a stronger property. Let us assume to this end that on its domain of definition  (usually a bounded subset of $\R^n$) function $B$ satisfies

\begin{equation}
\label{one}
\liminf_{|x|\to \infty}  \inf_{u_{j}}B\left(\frac{u_1}{1+|a_{1}\cdot x|},\ldots,\frac{u_{n}}{1+|a_{n}\cdot x|}\right)\geq 0\,.
\end{equation}
Here $\inf_{u_{j}}$ is taken over the ranges of the functions $u_{j}(\cdot)$ i.e., $\inf_{u_{j}}=\inf_{u_{1},\ldots, u_{n}\, :\, u_{j}\in \mathrm{range}(u_{j}(\cdot))}$.
Notice that  this is a property of $B$ and vectors $a_i$ and not just $B$ alone.

%Notice that if matrix $A$ has full rank $k$, then 
%\begin{equation}
%\label{max}
%\lim_{|x|\to\infty}\max_{i\in [1, \dots, k]} |a_i \cdot x|  =\infty\,.
%\end{equation}

%Condition \eqref{one} is what we will need to check in conjunction with \eqref{MMC} to have a sequence of isoperimetric inequalities satisfied. In fact, \eqref{max} is always true if we have a maximal rank of $A$, and this inequality and \eqref{one} give us the condition at infinity formulated in \eqref{infinity0}.

\bigskip

\noindent{\bf Example 1.}  We will be using (see Section \ref{Er}) such $B$:
$$
B(u_1, u_2, u_3):= u_3 -\Phi(\alpha_1\Phi^{-1}(u_1) +\alpha_2\Phi^{-1}(u_2))\,,
$$
(where  $\Phi(x):=\frac{1}{\sqrt{2\pi}}\int_{-\infty}^x e^{ -s^2/2} ds$)
with some constants $\alpha_1>0, \alpha_2>0$. The domain of definition will be cube $Q=[0,1-\delta]^3$ for any $0<\delta<1$.
Let $a_{1}=(1,0)^{T}$, $a_{2}=(0,1)^{T}$ and $a_{3}=(a_{31},a_{32})^{T}$ where $a_{31},a_{32}\neq 0$.
 Since we cannot find a vector $x \in \mathbb{R}^{2}$ which will be simultaneously orthogonal to $a_{1},a_{3}$ (or $a_{2},a_{3}$) then 
\eqref{one} is satisfied as
$$
\eps - \Phi(\alpha_1\Phi^{-1}(\eps) +\alpha_2\Phi^{-1}(u_2))\approx -\Phi(-\infty +\alpha_2\Phi^{-1}(u_2))=- \Phi(-\infty)=0\,,
$$
and symmetric claim holds for $u_2$. So the assumption at  infinity \eqref{infinity0} will follow.

\section{Simplifications and reductions of the Bellman equation of the second type: \eqref{MMC}}
\label{SR}

Let $D$ denote the diagonal matrix-function with $\nabla B=(B_1, \dots, B_n)$ on the diagonal. We usually assume that $A$ has a full rank. Now we assume that $AD$ has full rank in the domain of definition of $B$. In applications this routinely happens.  If $k=n$ then our condition  becomes trivial and is always true. 
If $k=n-1$ then $P_{\mathrm{ker}}$ has rank 1 and therefore  the first line of \eqref{MMC} holds if and only if 
\begin{align*}
&\mathrm{Tr}(P_{\mathrm{ker} \;AD}\left(A^{*}CA\bullet \mathrm{Hess}\, B\right)P_{\mathrm{ker} \;AD})=\\
&\sum_{j} B_{jj}\langle Ca_{j}, a_{j}\rangle - \sum_{i,j}B_{ij}B_{i}B_{j} \langle Ca_{i}, a_{j}\rangle \langle (AD^{2}A^{*})^{-1}a_{i},a_{j}\rangle \leq 0.
\end{align*} 

If $k=1$ then $d(A)\; \mathrm{Hess}\, B\; d(A) \leq 0$ on the orthogonal complement of the 1-dimensional space $(a_{1}B_{1},\ldots, a_{n}B_{n})$, where $d(A)$ is diagonal matrix having on the diagonal $a_{j}$.   For example, if $a_{j}=1$ for all $j$   this means that $\mathrm{Hess}\, B \leq 0$ on the  variable subspace orthogonal to  $\nabla B$.

\subsection{Applications to special $B$'s when $k=n-1$.}
Assume $k=n-1$. Assume also  $a_{j} =e_{j}$ for $j=1,\ldots, n-1$, where $e_{j}$ are basis vectors.    
%Now notice that 
%\begin{align*}
%\mathrm{Tr}(P_{\mathrm{ker} \;AD}\left(A^{*}CA\bullet \mathrm{Hess}\, B\right)P_{\mathrm{ker} \;AD})=
%\sum_{j} B_{jj}-B_{jj}B_{j}^{2}\langle (AD^{2}A^{*})^{-1}a_{j},a_{j} \rangle.
%\end{align*}
Since $(n-1)\times (n-1)$ matrix $AD^{2}A^{*}$ has the following form
\begin{align*}
AD^{2}A^{*}=B_{n}^{2}a_{n}a_{n}^{T}+D_{2}^{2}
\end{align*}
where $D_{2}$ is diagonal matrices consisting of elements $B_{1},\ldots, B_{n-1}$ on the diagonal,  we obtain  by Sherman--Morison formula
\begin{equation}
\label{ShMo}
(B_{n}^{2}a_{n}a_{n}^{T}+D_{2}^{2})^{-1}=D_{2}^{-2}-\frac{B_{n}^{2}D_{2}^{-2}a_{n}a_{n}^{T}D_{2}^{-2}}{1+B_{n}^{2}a_{n}^{T}D^{-2}_{2}a_{n}},
\end{equation}
therefore
\begin{align*}
&\mathrm{Tr}(P_{\mathrm{ker} \;AD}\left(A^{*}CA\bullet \mathrm{Hess}\, B\right)P_{\mathrm{ker} \;AD})= \sum \langle Ca_{j}, a_{j}\rangle (B_{jj}-B_{jj}B_{j}^{2}a_{j}^{T}D_{2}^{-2}a_{j})+\\
&\sum\langle Ca_{j}, a_{j}\rangle \frac{B_{jj}B_{j}^{2}B_{n}^{2}|a_{n}^{T}D_{2}^{-2}a_{j}|^{2}}{1+B_{n}^{2}a_{n}^{T}D^{-2}_{2}a_{n}}-\sum_{i\neq j}B_{ij}B_{i}B_{j} \langle Ca_{i}, a_{j}\rangle \langle (AD^{2}A^{*})^{-1}a_{i},a_{j}\rangle=\\
%&\frac{B_{11}-B_{11}B_{1}^{4}|a_{1}^{T}D_{2}^{-2}a_{1}|^{2}+B_{11}B_{1}^{4}|a_{1}^{T}D_{2}^{-2}a_{1}|^{2}+\sum_{j\geq 2}B_{jj}B_{j}^{2}B_{1}^{2}|a_{1}^{T}D_{2}^{-2}a_{j}|^{2} }{1+B_{1}^{2}a_{1}^{T}D^{-2}_{2}a_{1}}=\\
&\frac{1}{B_{n}^{-2}+a_{n}^{T}D_{2}^{-2}a_{n}}\left[\langle Ca_{n},a_{n}\rangle \frac{B_{nn}}{B_{n}^{2}}  +\sum_{j= 1}^{n-1} \langle Ca_{j},a_{j}\rangle \frac{B_{jj}}{B_{j}^{2}}a_{nj}^{2}\right]-\sum_{i\neq j}B_{ij}B_{i}B_{j} \langle Ca_{i}, a_{j}\rangle \langle (AD^{2}A^{*})^{-1}a_{i},a_{j}\rangle
\end{align*}

Notice that if $B(u_{1},\ldots, u_{n})=u_{n}-H(u_{1},\ldots, u_{n-1})$ then using \eqref{ShMo} again we get
\begin{equation}
\label{H2}
\mathrm{Tr}(P_{\mathrm{ker} \;AD}\left(A^{*}CA\bullet \mathrm{Hess}\, B\right)P_{\mathrm{ker} \;AD}) = \frac{-1}{1+a_{n}^{T}D_{2}^{-2}a_{n}}\left(\sum_{i,j=1}^{n-1}\frac{H_{ij}}{H_{i}H_{j}}a_{ni}a_{nj}c_{ij}\right).
\end{equation}

\bigskip

In fact,  $B_{nj}= B_{in}=0$, and \eqref{ShMo} gives us  for $i\neq j, i\neq n, j\neq n,$ 
\begin{align*}
&\langle (AD^{2}A^{*})^{-1}a_{i},a_{j}\rangle=
\left\langle  (B_{n}^{2}a_{n}a_{n}^{T}+D_{2}^{2})^{-1} a_i, a_j\right\rangle=
\\
&-\left\langle  \frac{B_{n}^{2}D_{2}^{-2}a_{n}a_{n}^{T}D_{2}^{-2}}{1+B_{n}^{2}a_{n}^{T}D^{-2}_{2}a_{n}}a_i, a_j\right\rangle = \frac{\La D_{2}^{-2}a_{n}a_{n}^{T}D_{2}^{-2}a_i, a_j\Ra}{B_n^{-2}+a_{n}^{T}D^{-2}_{2}a_{n}}=\frac{a_{ni}a_{nj}}{B_i^2B_j^2(B_n^{-2}+a_{n}^{T}D^{-2}_{2}a_{n})}\,,
\end{align*}
because $a_i=e_i, 1\leq i \leq n-1,$ and $D^{-2}$ is a diagonal matrix. Hence, \eqref{H2} is proved.

\bigskip

Provided that the condition at infinity  is satisfied for $B(u_1,\dots, u_n)= u_n-H(u_1,\dots, u_{n-1})$ and our vectors have the form: $a_1=e_{1}, a_2=e_2,\dots, a_{n-1}=e_{n-1}$,  some $a_{n}\in \R^{n-1}$,
we reduced  the Bellman equation of the second sort \eqref{MMC} to a following nonlinear partial differential ineaquality on $H$:

\begin{equation}
\label{onH}
\sum_{i,j=1}^{n-1}\frac{H_{ij}}{H_{i}H_{j}}a_{ni}a_{nj}c_{ij} \geq  0\,.
\end{equation}
%Of course the most interesting are the saturated solutions of this PD inequality.
\begin{cor}\label{preapp}
Given a vector $a_n=\{a_{nj}\}_{j=1}^k$, any  function $H$ for which there exists  $k\times k$ matrix $C\geq 0$  such that $\langle Ca_{n},a_{n}\rangle, c_{jj} >0$, and simultaneously the PD inequality \eqref{onH} holds, gives rise to an ``isoperimetric inequality" \eqref{Biso}:
$$
\text{If for all}\,\, x\in \R^k\,\, u_1(a_n\cdot x) -H(u_2(x_2),\dots, u_n(x_n))\ge 0\,\,\text{then}
$$
\begin{equation}
\label{H3iso}
\text{ for all}\,\, x\in \R^k,\, t>0,\,\,( P^{C}_t u_1)(a_n\cdot x) - H(P^{C}_t u_2(x_2),\dots, P^{C}_t u_n(x_n))\ge 0\,.
\end{equation}
(again, one should check the condition at infinity that appears from \eqref{infinity0}).
\end{cor}

\bigskip

\section{Further reductions in Bellman equation of the second type: PDE that rules the Prekopa--Leindler and Ehrhard inequality}
\label{Er}

Functional version of Prekopa--Leindler and Ehrhard's inequality in arbitrary dimension can be formulated as follows:  if 
\begin{align*}
\Phi^{-1}(h(\sum^{\ell} b_{j} x_{j})) \geq \sum^{\ell} b_{j} \Phi^{-1}(f_{j}(x_{j}))\quad \text{for all} \quad x_{j} \in \mathbb{R}^{m},
\end{align*}
then 
\begin{align*}
\Phi^{-1}\left(\int_{\mathbb{R}^{m}}hd\gamma_{m}\right) \geq \sum^{\ell} b_{j} \Phi^{-1}\left(\int_{\mathbb{R}^{m}}f_{j}d\gamma_{m}\right).
\end{align*}

In case of Ehrhard's inequality we require that $b_{j}>0$,  $\sum b_{j} \geq 1$, $b_{j}-\sum_{i\neq j} b_{i} \leq 1$, $h, f_{j} : \mathbb{R}^{m} \to [0,1]$ and $\Phi(x) = \int_{-\infty}^{x}d\gamma_{1}$, and in case of Prekopa--Leindler's inequality requirements are different: 
$b_{j}>0$,  $\sum b_{j} =1$, $h, f_{j} : \mathbb{R}^{m} \to \mathbb{R}^{+}$ and $\Phi(x) = e^{x}$.

We will prove these inequalities  by using second type of Bellman PDE when $m=1$. Arbitrary dimension follows easily by iterating  one dimensional case $m$ times.

\bigskip

Let $k=n-1$. Take any vector $a_{n}=b=(b_{1},\ldots, b_{k})\in \mathbb{R}^{k}$ such that $b_{j}>0$.  Take $H(x_{1},\ldots, x_{k})=\Phi(b_{1}\Phi^{-1}(x_{1})+\ldots +b_{k}\Phi^{-1}(x_{k}))$ where $\Phi(x)=\int_{-\infty}^{x}\varphi(x)dx$,  $\varphi>0$, $\varphi \in C^{1}$ and $\Phi(x)$ is finite for any $x \in \mathbb{R}$. Let $k\times k$ matrix $C=\{ c_{ij}\}\geq 0$. Clearly $\partial_{j} H > 0$ for all $j$. 

\bigskip
Direct computations give
\begin{align}\label{prepre}
\left(\sum_{i,j=1}^{k}\frac{H_{ij}}{H_{i}H_{j}}b_{i}b_{j}c_{ij}\right) = \frac{\varphi'(\sum b_{j}y_{j})}{\varphi(\sum b_{j}y_{j})}\langle Cb,b\rangle - \sum_{j} \frac{\varphi'(y_{j})}{\varphi(y_{j})}b_{j}c_{jj}.
\end{align}
Here $y_{j} = \Phi(x_{j})$ for all $j=1,\ldots, n-1$. 

\bigskip

Let us first require that $\langle Cb,b\rangle=1$ and $c_{jj}=1$ for all $j=1,\ldots, k$. (Of course the existence of such $C\ge 0$ should depend on vector $b$.)

\bigskip

In order to apply Corollary~\ref{preapp} we need to  have the following conditions:
\begin{itemize}
\item[A1.] There exists $C\geq 0$ such that $\langle Cb,b\rangle =1$ and $c_{jj}=1$ for all $j$. 
\item[A2.] Logarithmic derivative  of $\varphi$ satisfies the following ``concavity condition'':
\begin{equation}
\label{logdev}
(\log \varphi)'(\Sigma b_{j}y_{j}) \ge \Sigma b_{j} (\log\varphi)'(y_j)\,.
\end{equation}
\item[A3.] Condition at infinity (\ref{infinity0}) is satisfied i.e., 
\begin{align*}
\liminf_{|x|\to \infty} \inf_{u_{j}\geq 0} \frac{u_{n}}{1+|b\cdot x|}-\Phi\left(b_{1}\Phi^{-1}\left(\frac{u_{1}}{1+|x_{1}|}\right)+\ldots +b_{k}\Phi^{-1}\left(\frac{u_{k}}{1+|x_{k}|}\right)\right)\geq 0.
\end{align*}
\end{itemize}
Then under the assumptions A1-A3, Corollary~\ref{preapp} implies: if for compactly supported functions $u_{1},\ldots, u_{n}$ we have 
\begin{align*}
u_{n}(b\cdot x)\geq \Phi(b_{1}\Phi^{-1}(u_{1}(x_{1}))+\ldots +b_{k}\Phi^{-1}(u_{k}(x_{k}))) \quad \text{for all} \quad x \in \mathbb{R}^{k},
\end{align*}
then 
\begin{align*}
\int u_{n} d\gamma \geq \Phi\left[ b_{1}\Phi^{-1}\left(\int u_{1}d\gamma \right)+\ldots +b_{k}\Phi^{-1}\left(\int u_{k}d\gamma\right)\right].
\end{align*}

We are going to study each condition of $A1-A3$ separately. 

\bigskip

{\bf Condition A1.} Since $C=V^{*}V$ for some $V=(v_{1},\ldots, v_{k})$ where $v_{j}$ are columns of $V$ we see that condition $c_{jj}=\langle Ce_{j},e_{j}\rangle=1$ implies that $v_{j}$'s are unit vectors. Condition $\langle Cb,b\rangle=1$ implies that $|\sum_{j=1}^{k} v_{j}b_{j}|=1$. The last one gives necessary conditions $\sum_{j} b_{j}\geq 1$. Note also that triangle inequality together with $|\sum_{j=1}^{k} v_{j}b_{j}|=1$ implies that $1\geq b_{j}-\sum_{i\ne j}b_{i}$. Thus we obtain two necessary conditions 
\begin{align}\label{bartha}
1\leq \sum_{j=1}^{k}b_{j} \quad \text{and}\quad 1\geq b_{j}-\sum_{i\neq j}b_{i}\quad \text{for all} \quad j=1,\ldots, k. 
\end{align}
It turns out that these two conditions are also sufficient for the existence of matrix $C$ (see Lemma~3 in \cite{barthe1}).

\bigskip

{\bf Condition A2.} This condition will be just assumption on the function $\varphi$.  Note that the particular function $\varphi(x)=\lambda e^{-rx^{2}}$ always gives  us equality   in (\ref{logdev}). In particular, this choice will give us Ehrhard's inequality, namely the choice  $\varphi(x)=\frac{1}{\sqrt{2\pi}} e^{-x^{2}/2}$. 

Another interesting choice  is $\varphi(x)=e^{x}$ for which (\ref{logdev}) becomes $\sum b_j\leq 1$. This together with (\ref{bartha}) gives $\sum b_{j}=1$. This will give us Prekopa--Leindler's inequality. Note that in this case  by our choice of $\phi$ we have  $H(x_{1},\ldots, x_{k})=x_{1}^{b_{1}}\cdots x_{k}^{b_{k}}$, where $b_{j}>0$ and $\sum_{j} b_{j}=1$. 

\bigskip

{\bf Condition A3.}
Here we follow the same reasonings as in Example 1. Note that if $x \to \infty$ then by compactness argument,  we can choose subsequence and we can assume that one of the coordinates $x_{s} \to \infty$. This implies that $\Phi^{-1}\left(\frac{u_{s}}{1+|x_{s}|}\right)\approx -\infty$. Thus condition A3 would be satisfied provided that none of the $\Phi^{-1}\left(\frac{u_{j}}{1+|x_{j}|}\right) \approx \infty$, $j\neq s$. For this purpose let us require that $u_{i}$'s are separated from infinity in the sense of $\Phi$, i.e.,  $\Phi^{-1}(|u_{j}|_{\infty})<\infty$.

\bigskip

It is interesting to mention that we can get rid off Condition A1, because based on Corollary~\ref{preapp} we do not have to require $\langle Ca_{j},a_{j}\rangle =1$ for all $j$ (this was necessary for the applications). In general the above considerations lead us to the following corollary.

Let $k\geq 2$. $b=(b_{1},\ldots, b_{k})$ and $b_{j}>0$. For $\varphi>0$ such that $\varphi \in C^{1}$ set $\Phi(x)=\int_{-\infty}^{x} \varphi$. Assume $\Phi(x)$ is locally finite.
Let $u_{j}$ be smooth functions with compact support such that $\Phi^{-1}(|u_{j}|_{\infty})<\infty$.

\begin{cor}
Let  $v_{j} \in \mathbb{R}^{k}$ be such that $v_{j} \neq 0$ and $\sum b_{j} v_{j} \neq 0$.
If $\Phi(x)=\int_{-\infty}^{x} \varphi$, and
\begin{align*}
|\sum b_{j}v_{j}|^{2}(\log \varphi)'(\sum b_{j}y_{j}) \ge \sum b_{j}|v_{j}|^{2} (\log\varphi)'(y_j) \quad \text{for all}\quad y_{j}, 
\end{align*}
then the inequality 
\begin{align*}
\Phi^{-1}\left(u_{0}(\sum b_{j}x_{j}) \right)\geq \sum b_{j}\Phi^{-1}(u_{j}(x_{j})) \quad \text{for all} \quad x \in \mathbb{R}^{k},
\end{align*}
implies 
\begin{align*}
\Phi^{-1}\left(\int_{\mathbb{R}} u_{0}(|\sum b_{j}v_{j}|y) d\gamma(y)\right) \geq \sum  b_{j}\Phi^{-1}\left(\int_{\mathbb{R}} u_{j}(|v_{j}|y)d\gamma(y) \right).
\end{align*}
\end{cor}
\begin{proof}
Corollary immediately follows by taking $C=V^{T}V$ where $V=(v_{1},\ldots,v_{k})$ and noticing that $\langle Ce_{j},e_{j}\rangle=|v_{j}|^{2}$  and 
$\langle Cb,b\rangle= |\sum b_{j} v_{j}|^{2}$. 
\end{proof}

\begin{cor}\label{corgla}
Let $\Omega$ be a bounded  rectangular domain in $\mathbb{R}^{2}$ such that $0 \in \mathrm{Cl}(\Omega)$. Let $\alpha, \beta \in \mathbb{R}$ be such that $|\alpha|+|\beta| \geq $ and $1 \geq ||\alpha| - |\beta||$. Let a smooth function $H (x,y) : \Omega \to \mathbb{R}$ be such that $\frac{\partial H}{\partial x}, \frac{\partial H}{\partial y} \neq 0$ and 
\begin{align}\label{glavnoe}
(1-\alpha^{2}-\beta^{2})\frac{\partial H}{\partial x} \frac{\partial H}{\partial y} \frac{\partial^{2} H}{\partial x \partial y}+
\alpha^{2} \left( \frac{\partial H}{\partial y} \right)^{2}\frac{\partial^{2} H}{\partial x^{2}}+\beta^{2} \left(\frac{\partial H}{\partial x}\right)^{2}\frac{\partial^{2} H}{\partial y^{2}}\geq 0.
\end{align}

Then for smooth bounded function $u_{3}$ and compactly supported functions $u_{1}, u_{2}$  such that $(u_{1},u_{2}) :\mathbb{R}^{k} \to \Omega$,  the inequality holds 
\begin{align*}
\int_{\mathbb{R}^{k}}u_{3}d\gamma_{k} \geq H\left(\int_{\mathbb{R}^{k}} u_{1} d\gamma_{k}, \int_{\mathbb{R}^{k}} u_{2} d\gamma_{k}\right)
\end{align*}
whenever $u_{3}(\alpha x + \beta y) \geq H(u_{1} (x), u_{2}(y))\quad \text{for all}\quad x,y \in \mathbb{R}^{k}$.
\end{cor}
\begin{proof}
The corollary is immediate consequence of Corollary~\ref{preapp}. Indeed, take $n=3$ and $k=1$. Take $a_{3}=(\alpha, \beta)$. It is clear that we should choose $c_{11}=c_{22}=1$ and $c_{12}=\frac{1-\alpha^{2}-\beta^{2}}{4\alpha^{2}\beta^{2}}$. In this case condition $C\geq 0$ is the same as $|\alpha|+|\beta| \geq $ and $1 \geq ||\alpha| - |\beta||$. Inequality (\ref{glavnoe}) is the same as (\ref{onH}).

Now we left to check condition at infinity~\ref{infinity0}. Let $|(x,y)| \to \infty$. 
Suppose that  both $|x|, |y| \to \infty$. 
Since $u_{1}$ and $u_{2}$ are compactly supported this means that we need to ensure that the following inequality holds 
\begin{align*}
u_{3}(z, t) \geq H(0, 0)
\end{align*}
for all $z \in \mathbb{R}^{k}$. This follows from the pointwise inequality: since  $u_{3}(\alpha x + \beta y) \geq H(u_{1} (x), u_{2}(y))$ then taking $x,y$ sufficiently large we can make $\alpha x+ \beta y$ to be any point $z  \in \mathbb{R}^{k}$. Then from the pointwise inequality $u_{3}(y) \geq H(0,0)$  for all $y \in \mathbb{R}^{k}$ we obtain integral inequality  after integrating it  with respect to probability measure $p_{t}(z,y)dy$ of the heat semigroup $P_{t}$.  

Now consider the case when $|y| \to \infty$ and $|x|$ is bounded. In this case we need to show that 
\begin{align*}
\liminf_{|y| \to \infty }u_{3}(y,t)\geq H(u_{1}(x,t),0).
\end{align*}
Notice that pointwise inequality implies $\liminf_{|y|\to \infty} u_{3}(y) \geq H(u_{1}(x),0)$. Since $u_{1}(x,t)=u_{1}(x^{*})$ for some $x^{*}$  and 
$\liminf_{|y| \to \infty} u_{3}(y,t)\geq \liminf_{|y| \to \infty} u_{3}(y)$ we obtain the desired result.

 In order to obtain the corollary for the general rank case i.e., for arbitrary $k>1$ we can iterate the inequality as we did before in case of Ehrhard's inequality (or one can see Section~\ref{grank}).
\end{proof}

\subsection{Solving particular case of second type PDE}
In this section we will partially solve PDE (\ref{onH}) in the case $n=3$. assume that $a_{3}=(a,b)\in \mathbb{R}^{2}$. Let us require that for $C\geq 0$ we have  $c_{jj}=1$ and $\langle Ca_{3},a_{3}\rangle =1$. This can happen if and only if $|a|+|b|\geq 1$ and $||a|-|b||\leq 1$. In other words, this can be written as one condition:
\begin{align*}
1\geq \frac{(1-a^{2}-b^{2})^{2}}{4a^{2}b^{2}}.
\end{align*}
Condition $\langle Ca_{3},a_{3}\rangle=1$ implies that $c_{12}=\frac{1-a^{2}-b^{2}}{2ab}$. Therefore (\ref{onH}) takes the following form 
\begin{align}\label{diffusion}
H_{1}H_{2}H_{12}(1-a^{2}-b^{2})+a^{2}H_{2}^{2}H_{11}+b^{2}H_{1}^{2}H_{22}=0,
\end{align}
where $H=H(x,y)$. Let us show that the equation (\ref{diffusion}) can be reduced to second order linear differential equation with constant coefficients.

 In particular, we will see that if 
$|c_{12}|=1$ then the equation becomes parabolic equation and it reduces to heat equation, and if $|c_{12}|<1$ then the equation becomes elliptic equation which reduces to  Laplacian eigenfunctions.
\bigskip

Let $H_{1}=p, H_{2}=q$ (therefore $p_{y}=q_{x}$), then (\ref{diffusion}) becomes 
\begin{align*}
pqp_{y}(1-a^{2}-b^{2})+a^{2}q^{2}p_{x}+b^{2}p^{2}q_{y}=0.
\end{align*}

Assuming that the map $(x,y)\to (p(x,y),q(x,y))$ is locally invertible we obtain (exactly as we did in Section~\ref{BrSec}) 
\begin{align}\label{bra}
-pqx_{q}(1-a^{2}-b^{2})+a^{2}q^{2}y_{q}+b^{2}p^{2}x_{p}=0,
\end{align}
and $x_{q}=y_{p}$. 

We differentiate (\ref{bra}) with respect to $p$: 
\begin{align*}
-qx_{q}(1-a^{2}-b^{2})-pqx_{pq}(1-a^{2}-b^{2})+a^{2}q^{2}x_{qq}+2b^{2}px_{p}+b^{2}p^{2}x_{pp}=0.
\end{align*}
Let $x(p,q)=B(\ln p, \ln q)$ then  $p^{2}x_{pp}=B_{uu}-B_{u}$ and $q^{2}x_{qq}=B_{vv}-B_{v}$. Then 
\begin{align*}
-B_{v}(1-a^{2}-b^{2})-B_{uv}(1-a^{2}-b^{2})+a^{2}(B_{vv}-B_{v})+2b^{2}B_{u}+b^{2}(B_{uu}-B_{u})=0.
\end{align*}
Hence
\begin{align}\label{PDE1}
b^{2}B_{uu} +B_{uv}(1-a^{2}-b^{2})+a^{2}B_{vv}+B_{u}b^{2}+B_{v}(b^{2}-1)=0.
\end{align}
Thus we obtained second order linear PDE with constant coefficients. All we know about the numbers $a,b$ is that 
\begin{align*}
1 \geq \frac{(1-a^{2}-b^{2})^{2}}{4a^{2}b^{2}}
\end{align*}
So if $4a^{2}b^{2}= (1-a^{2}-b^{2})^{2}$ (which is the same as $|c_{12}|=1$) then the above equation corresponds to the parabolic equation, and if $4a^{2}b^{2}> (1-a^{2}-b^{2})^{2}$ then the above equation becomes elliptic equation. 

{\bf Parabolic equation: heat equation.}
Assume that $4a^{2}b^{2}= (1-a^{2}-b^{2})^{2}$, i.e., $b=1-a$.
Then our PDE becomes 
\begin{align*}
(1-a)^{2}B_{uu}+2a(1-a)B_{uv}+a^{2}B_{vv}+(1-a)^{2}B_{u}+a(a-2)B_{v}=0.
\end{align*}
Since this corresponds to  parabolic equation we can not make coefficient in front of $B_{uu}$ and $B_{vv}$ zero simultaneously. 
So we make the following change of variables $B(u,v)=M(\frac{a}{1-a}u-v,u)$. Then 
\begin{align*}
M_{22}+M_{2}+\frac{a(3-2a)}{(1-a)^{2}}M_{1}=0.
\end{align*}
The following technical lemma describes solutions of this PDE.
\begin{lemma}
If 
\begin{align*}
M_{22}+c_1M_2+c_2M_1=0
\end{align*}
and $c_2\neq 0$ then $M(x,y)=e^{-\frac{c_{1} y}{2}+\frac{c_{1}^{2}x}{4c_{2}}}W(\frac{-x}{c_{2}},y)$ where $W$ satisfies the heat equation $W_{22}=W_{1}$.
\end{lemma}

{\bf Elliptic equation: Laplacian eigenfunctions.}
In order to get rid off mixed derivatives we make change of variables as follows
$$
B(u,v)=M\left(u\frac{(1-a^{2}-b^{2})}{2b^{2}}-v, u\frac{\sqrt{4a^{2}b^{2}-(1-a^{2}-b^{2})^{2}}}{2b^{2}}\right).
$$
Then the equation (\ref{PDE1}) becomes
\begin{align*}
\Delta M+M_{2}\frac{2b^{2}}{\sqrt{4a^{2}b^{2}-(1-a^{2}-b^{2})^{2}}}+M_{1}\frac{2b^{2}(3-3b^{2}-a^{2})}{4a^{2}b^{2}-(1-a^{2}-b^{2})^{2}}=0.
\end{align*}

The following technical lemma reduces the question to Laplacian eigenfunction problem: 
\begin{lemma}
If
\begin{align*}
M_{11}+M_{22}+c_{1}M_{1}+c_{2}M_{2}=0,
\end{align*}
then $M(x,y)=e^{-\frac{c_{1}x}{2}-\frac{c_{2}y}{2}}W(x,y)$ where $W$ is eigenfunction of the Laplacian, i.e., 
\begin{align*}
\Delta W = \left( \frac{c_{1}^{2}+c_{2}^{2}}{4}\right) W.
\end{align*} 
\end{lemma}
\section{General rank case}\label{grank}
\subsection{Special initial data}
So far we were considering special initial datas of the form $f(x)=F(a\cdot x)$ where $x,a \in \mathbb{R}^{k}$. It is natural to consider the following initial datas as well $f(x)=F(xA)$ where $A$ is $k\times s$ matrix and $x \in \mathbb{R}^{k}$, and $F :\mathbb{R}^{s}\to \mathbb{R}$. Note that we are writing $xA$ instead of more usual notations $Ax$ only because we want to keep the same notations as above $a\cdot x=xa$ where $a$ was column and $x$ was a row. 

For these initial datas absolutely nothing changes except we will work with larger matrices. Let us briefly formulate all results and leave the details. For a symmetric matrix $Q=\{q_{ij}\}$ we set $L_{Q}=\sum q_{ij} \frac{\partial^{2}}{\partial x_{i} \partial x_{j}}$. Corresponding semigroup will be denoted by $P^{Q}_{t}$. 
Further everywhere $C$ is symmetric $k\times k$ matrix. 
Analog of $1D$ heat flow is (see Section~\ref{spid})
\begin{align*}
\frac{\partial }{\partial t} U(y,t)=L_{A^{*}CA} U(y,t), \quad y \in \mathbb{R}^{s}, \quad U(y,0)=F(y). 
\end{align*}
Then 
\begin{align*}
P_{t}^{C}f(x)=U(xA,t)=\int_{\mathbb{R}^{s}}F(xA+(2tA^{*}CA)^{1/2}y)d\gamma_{s}(y). 
\end{align*}
So in order the expressions to be justified we only need to require $A^{*}CA>0$ but we do not need  $C>0$. Note that 
\begin{align*}
(L_{C}-\partial_{t})U(xA,t)=(L_{A^{*}CA}-\partial_{t})U(y,t)|_{y=xA}=0.
\end{align*}

 Therefore further we will be using sometimes the notation $P_{t}^{C}f(x)$ even though $C$ is not necessarily positive however we will assume that $A^{*}CA>0$. 
Note that 
\begin{align*}
\nabla P_{t}^{C}f(x)=A(\nabla_{y}U(y,t)|_{y=xA})^{T}.
\end{align*}
\subsection{First type of Bellman PDE for the general rank case.}\label{gsec1}
Let $A_{1},\ldots, A_{n}$ be matrices such tat $A_{j}$ is $k\times k_{j}$ size and let $A=(A_{1},\ldots, A_{n})$ be $k\times (k_{1}+\ldots+ k_{n})$ size. Let $B: \Omega\subset\mathbb{R}^{k} \to \mathbb{R}$ be smooth function on some rectangular domain $\Omega$. Take any $k\times k$ symmetric matrix  $C>0$. Let $u_{j} :\mathbb{R}^{k_{j}}\to \mathbb{R}$ be smooth compactly supported functions, and let $\vec{u}(x)=(u_{1}(xA_{1}),\ldots, u_{n}(xA_{n})): \mathbb{R}^{k} \to \Omega$. 
\begin{theorem}\label{gpde1}
The following conditions are equivalent:
\begin{itemize}
\item[(i)] $A^{*}CA\bullet \Hess B \leq 0$ on $\Omega$. 
\item[(ii)] $(P_{t}^{C}B(\vec{u}))(x)\leq B((P_{t}^{C}\vec{u})(x))$ for all $t\geq 0$, $x\in \mathbb{R}^{k}$ and $u_{j}$.
\item[(iii)] $(P_{t}^{C}B(\vec{u}))(x)\leq B((P_{t}^{C}\vec{u})(x))$ for $t=1/2$, $x=0$ and for all  $u_{j}$.
\end{itemize}
\end{theorem}
Here $A^{*}CA\bullet \Hess B$ denotes $(\sum k_{j})\times (\sum k_{j})$  matrix $\{ A_{i}^{*}CA_{j}\partial_{ij}B\}_{i,j=1}^{n}$ i.e., $A^{*}CA\bullet \Hess B$ is constructed by the bloks $A_{i}^{*}CA_{j}\partial_{ij}B$. Note that if $C$ and $A^{*}_{j}A_{j}$ are  identity matrices then  condition (iii) of Theorem~\ref{gpde1} takes the form 
\begin{align*}
\int_{\mathbb{R}^{k}}B(u_{1}(xA_{1}),\ldots, u_{n}(xA_{n}))d\gamma_{k}(x)\leq B\left(\int_{\mathbb{R}^{k_{1}}}u_{1}(x) d\gamma_{k_{1}}(x),\ldots, \int_{\mathbb{R}^{k_{n}}}u_{n}(x) d\gamma_{k_{n}}(x) \right).
\end{align*}
\subsection{Second type of Bellman PDE for the general rank case.}
We use the same notations as in  the previous section except instead of $C>0$ we only assume that $C\geq 0$ and $A_{j}^{*}CA_{j}>0$. 
Let $T=(B_{1}A_{1},\ldots, B_{n}A_{n})$ be $k\times(k_{1}+\ldots+k_{n})$ matrix, where $B_{j}=\partial_{j} B$. 
\begin{theorem}
Assume $P_{\ker T}(A^{*}CA\bullet \Hess B)P_{\ker T}\leq 0$.  Then 
\begin{align*}
\text{if}\quad &B((P_{t}^{C}\vec{u})(x))\geq 0\quad \text{for} \quad t=0\quad \text{and}\quad  \forall x \in \mathbb{R}^{k},\\
\text{then}\quad &B((P_{t}^{C}\vec{u})(x))\geq 0\quad \text{for} \quad t\geq 0 \quad \text{and}\quad  \forall x \in \mathbb{R}^{k},
\end{align*}
provided that condition at infinity holds:
\begin{align*}
\liminf_{|x| \to \infty}\inf_{u_{j}} B\left(\frac{u_{1}}{1+|xA_{1}|},\ldots, \frac{u_{n}}{1+|xA_{n}|}\right)\geq 0.
\end{align*} 
\end{theorem}
\subsection{Applications tensorizes}
We remind that Borell's Gaussian noise stability (see Section~\ref{BGNS})  and hypercontractivity of Ornstein--Uhlenbeck (see Section~\ref{hyper}) were consequences of inequality (\ref{fg1}) which in turn is equivalent to PDE (\ref{Bryant}). Let us show that the same function implies these results in arbitrary dimension. Namely it is enough to show that if $B$ satisfies (\ref{Bryant}) then 
\begin{align*}
\int_{\mathbb{R}^{2n}}B(f(x),g(px+\sqrt{1-p^{2}}y))d\gamma_{2}(x,y)\leq B\left( \int_{\mathbb{R}^{n}}f(x)d\gamma_{n}(x),\int_{\mathbb{R}^{n}}g(x)d\gamma_{n}(x)\right).
\end{align*}
Indeed, we will apply Theorem~\ref{gpde1} for $A_{1}=(I_{n\times n},0_{n\times n})^{T}$, $A_{2}=(pI_{n\times n}, \sqrt{1-p^{2}} I_{n\times n})^{T}$ and $C=I_{n\times n}$. Here $I_{n\times n}$ is $n\times n$ identity matrix and $0_{n\times n}$ is $n \times n$ zero matrix. Then 
\begin{align*}
A_{1}^{*}A_{1}=A_{2}^{*}A_{2}=I_{n\times n} \quad \text{and}\quad  A_{1}^{*}A_{2}=A_{2}^{*}A_{1}=pI_{n\times n}.
\end{align*}
Therefore condition $A^{*}CA \bullet \Hess B\leq 0$ becomes 
\begin{align*}
            \left( {\begin{array}{cc}
             B_{11}  & pB_{12}  \\
             pB_{12}  & B_{22}               
                \end{array} } \right)\otimes I_{n} \leq 0, 
\end{align*}
and this is equivalent to (\ref{Bryant})

\bigskip
\bigskip

\section{Short review of some classical isoperimetric inequalities}\label{review}
\subsection*{Brunn--Minkowski and isoperimetric inequalities}
Let $A$ and $B$ be nonempty compact subsets of $\mathbb{R}^{n}$. 
\begin{thmm}
The following sharp Brunn--Minkowski inequality holds 
\begin{align*}
|A+B|^{1/n}\geq |A|^{1/n}+|B|^{1/n},
\end{align*}
where $n\geq 1$ and $|A|$ denotes Lebesgue measure of the set $A$. 
\end{thmm}
The Brunn-Minkowski inequality is a consequence of  its multiplicative version: 
\begin{thmm}
Let $\lambda \in (0,1)$. Then for any compact measurable sets $U, V \subset \mathbb{R}^{n}$ we have 
\begin{align}\label{brunmm}
|\lambda U+(1-\lambda)V| \geq  |U|^{\lambda}|V|^{1-\lambda}.
\end{align}
\end{thmm}
Indeed, if one sets $U\lambda =A$ and $(1-\lambda)V=B$ then inequality (\ref{brunmm}) takes the form 
\begin{align}\label{multiplic}
|A+B| \geq  \frac{|A|^{\lambda}|B|^{1-\lambda}}{\lambda^{\lambda n}(1-\lambda)^{(1-\lambda)n}}.
\end{align}
By maximizing the right hand side of (\ref{multiplic}) over $\lambda \in (0,1)$ we obtain the Brunn--Minkowski inequality. 

Brunn--Minkowski inequality implies the classical isoperimetric inequality:
\begin{thmm}
Among all simple closed surfaces with given surface area, the sphere encloses a region of maximal volume. In other words 
\begin{align*}
|\partial A| \geq n |A|^{1-\frac{1}{n}}|B(0,1)|^{\frac{1}{n}}.
\end{align*}
Where $|\partial A|$ means surface area of the boundary of the body $A$. $|A|$ denotes volume of the body and $B(0,1)$ denotes the ball of radius $1$ at center $0$. 
\end{thmm}
Indeed, let us sketch the proof: Since $|A+B(0,\varepsilon)| = |A|+\eps |\partial A| +O(\varepsilon^{2})$, we have 
\begin{align*}
|\partial A|=\lim_{\eps\to 0} \frac{|A+B(0,\varepsilon)|-|A|}{\eps}\geq \lim_{\eps \to 0} \frac{(|A|^{1/n}+|B(0,\eps)|^{1/n})^{n}-|A|}{\eps}=n|A|^{1-\frac{1}{n}}|B(0,1)|^{\frac{1}{n}}.\\
\end{align*}

For the possible references we refer the reader to \cite{Ball1, barthe2, Gard}
\subsection*{Sobolev inequality}
It is  known that the classical isoperimetric inequality is equivalent to its functional version, to  Sobolev inequality on $\mathbb{R}^{n}$ with optimal constant 
\begin{align}\label{sobolev}
\int_{\mathbb{R}^{n}}|\nabla f| \geq n | B(0,1)|^{\frac{1}{n}}\left( \int_{\mathbb{R}^{n}}|f|^{\frac{n}{n-1}}\right)^{1-\frac{1}{n}}.
\end{align}
 Indeed, testing (\ref{sobolev})  over characteristic functions   $f(x)={\bf 1}_{A}(x)$ we obtain implication in one direction. Opposite direction follows from Coarea formula: assume $f \geq 0$ is sufficiently nice compactly supported function. Then by coarea formula we have 
\begin{align*}
\int_{\mathbb{R}^{n}}|\nabla f| dx = \int_{0}^{\infty}|\{ x\, :\; f(x)=t\}|dt\geq n |B(0,1)|^{\frac{1}{n}}\int_{0}^{\infty}|\{ x\, :\; f(x)\geq t\}|^{1-\frac{1}{n}}dt.
\end{align*}
It is left to show that 
\begin{align*}
\left(\int_{0}^{\infty}|\{ x\, :\; f(x)\geq t\}|^{\frac{n-1}{n}}dt\right)^{\frac{n}{n-1}}\geq  \frac{n}{n-1}\int_{0}^{\infty}|\{x\, :\; f(x) \geq t\}| t^{\frac{1}{n-1}}dt
\end{align*}
This follows from the following observation 
\begin{align*}
F\left(\int_{0}^{\infty} \varphi \right)=\int_{0}^{\infty} \frac{d}{dt}F\left(\int_{0}^{t}\varphi \right)dt=\int_{0}^{\infty}F'\left(\int_{0}^{t}\varphi\right) \varphi dt\geq \int_{0}^{\infty}F'(t\varphi(t))\varphi(t)dt,
\end{align*}
where  $\varphi$ is decreasing and $F'$ is increasing ($F(t)=t^{\frac{n}{n-1}}$, $\varphi(t)=|\{ x\, :\; f(x)\geq t\}|^{\frac{n-1}{n}}$). So the claim follows.

\subsection*{Prekopa--Leindler inequality}
Multiplicative Brunn--Minkowski inequality follows from its functional version, so called Prekopa--Leindler inequality. 
\begin{thmm}
Let $h,f,g$ be positive measurable functions and $\lambda \in (0,1)$. If 
\begin{align}\label{domination}
h(\lambda x + (1-\lambda) y) \geq f(x)^{\lambda}g(y)^{1-\lambda}
\end{align}
Then 
\begin{align*}
\int_{\mathbb{R}^{n}}h \geq \left( \int_{\mathbb{R}^{n}}f\right)^{\lambda} \left(\int_{\mathbb{R}^{n}} g\right)^{1-\lambda}.
\end{align*}
\end{thmm}
If one takes $h(x) = \bf{1}_{\lambda U+(1-\lambda)V}(x), f(x) = \bf{1}_{U}(x)$ and $g(x)=\bf{1}_{V}(x)$ then clearly the assumption (\ref{domination}) is satisfied and  one obtains multiplicative version of Brunn--Minkowski inequality.

Straightforward generalization of Prekopa--Leindler inequality takes the following form: 
\begin{thmm}
Let $f_{j} :  \mathbb{R}^{n} \to R_{+}$ be integrable functions, and let $\sum_{j=1}^{m}\lambda_{j}=1$, $0<\lambda_{j} < 1$. If 
\begin{align*}
h\left(\sum_{j=1}^{m}\lambda_{j} x_{j}\right)\geq \prod_{j=1}^{m} f(x_{j})^{\lambda_{j}},
\end{align*}
then 
\begin{align*}
\int_{\mathbb{R}^{n}}h \geq \prod_{j=1}^{m} \left(\int_{\mathbb{R}^{n}}f_{j} \right)^{\lambda_{j}}.
\end{align*}
\end{thmm}

The above inequality can be treated as reverse to H\"older's inequality:
\begin{align*}
\int_{\mathbb{R}^{n}}\sup\left\{\prod_{j=1}^{m} f(x_{j})^{\lambda_{j}}\, :\; \sum x_{j}\lambda_{j}=z\right\}dz \geq \prod_{j=1}^{m} \left(\int_{\mathbb{R}^{n}}f_{j} \right)^{\lambda_{j}} \geq \prod_{j=1}^{m} \int_{\mathbb{R}^{n}}f_{j}(x_{j})^{\lambda_{j}}.
\end{align*}
where integral in the left hand side is understood as upper Lebesgue integral. 

Note that we proved Prekopa--Leindler inequality in Section~\ref{Er} when $\Phi(x)=e^{x}$ (see discussions given after the explanation of {\em Condition A2}). Basically the reason inequality holds is because the function $H(x_{1}, \ldots, x_{m})=\prod_{j=1}^{m} x_{j}^{\lambda_{j}}$ satisfies partial differential inequality (\ref{onH}) for appropriate choice of $C$ and $a_{n} = (\lambda_{1},\ldots,\lambda_{m})$. 

One of the other applications of Prekopa--Leindler inequality in probability is that: 
\begin{corr}
 If  $F(x,y) : \mathbb{R}^{n}\times \mathbb{R}^{m} \to \mathbb{R}^{+}$ is log-concave distribution i.e., 
\begin{align*}
F(\lambda u + (1-\lambda)v) \geq F(u)^{1-\lambda}F(v)^{\lambda}\quad \text{for all}\quad u, v \in \mathbb{R}^{n+m},
\end{align*}
then $H(x)=\int_{\mathbb{R}^{m}}F(x,y)dy$ is log-concave distribution. 
\end{corr}
The corollary immediately follows from application of Prekopa--Leindler inequality to the functions $F(x,\lambda y_{1}+(1-\lambda)y_{2}), F(x, y_{1})$ and $F(x,y_{2})$. 

\subsection*{Borell--Brascamp--Lieb inequality}
We also mention Borell--Brascamp--Lieb inequality since it generalizes Prekopa--Leindler inequality 
\begin{thmm}
Let $h,f,g$ be nonnegative functions, $0<\lambda<1$ and $-\frac{1}{n}\leq p \leq \infty$. Suppose 
\begin{align*}
h(\lambda x+ (1-\lambda)y)\geq M_{p}(f(x),g(y),\lambda), 
\end{align*}
where 

\begin{align*}
M_{p}(a,b,\lambda):= (\lambda a^{p}+(1-\lambda)b^{p})^{1/p}, \quad M_{0} := (a,b,\lambda)= a^{\lambda}b^{1-\lambda}.
\end{align*}
Then 
\begin{align}\label{bbrin}
\int_{\mathbb{R}^{n}}h \geq M_{\frac{p}{np+1}}\left(\int_{\mathbb{R}^{n}}f, \int_{\mathbb{R}^{n}}g, \lambda \right).
\end{align}
\end{thmm}

\bigskip 
Notice that $H(x,y)=M_{p}(x,y,\lambda)$ satisfies partial differential inequality (\ref{glavnoe}) for $p\geq 1$ (here $(\alpha, \beta)=(\lambda, 1-\lambda)$). Indeed, 
\begin{align*}
&(1-\alpha^{2}-\beta^{2})\frac{\partial H}{\partial x} \frac{\partial H}{\partial y} \frac{\partial^{2} H}{\partial x \partial y}+
\alpha^{2} \left( \frac{\partial H}{\partial y} \right)^{2}\frac{\partial^{2} H}{\partial x^{2}}+\beta^{2} \left(\frac{\partial H}{\partial x}\right)^{2}\frac{\partial^{2} H}{\partial y^{2}}=\\
&(p-1)\frac{\lambda(1-\lambda)(x^{p}-y^{p})^{2}}{(xy)^{p} H(x,y)}\geq 0
\end{align*}

Thus by Corollary~\ref{corgla} we obtain 
\begin{align}\label{corbor}
\int_{\mathbb{R}^{n}}h \geq M_{p}\left(\int_{\mathbb{R}^{n}}f, \int_{\mathbb{R}^{n}}g, \lambda \right).
\end{align}

 Also notice that $M_{p}(x,y,\lambda)\geq M_{\frac{p}{np+1}}(x,y,\lambda)$ for $x,y\geq 0$ and $-\frac{1}{n}<p <\infty$ (this is a direct computation: by homogeneity we can assume that $x=1$, and the rest follows by showing that the derivative of the function $f(y)=(\lambda+(1-\lambda)y^{p})^{1/p}- (\lambda +(1-\lambda)y^{\frac{p}{np+1}})^{\frac{np+1}{p}}$ has only one root $y=1$).

Thus inequality (\ref{corbor}) is better than (\ref{bbrin}), and hence  it implies Borell--Brascamp--Lieb inequality in case $p\geq 1$. 

In the case $-\frac{1}{n}\leq p \leq 1$ we do not know how to derive Borell--Brascamp--Lieb inequality by using Corollary~\ref{corgla}. The reason is because the inequality (\ref{glavnoe}) does not hold if $p<1$. 

\subsection*{Ehrhard's inequality}
The condition of Prekopa--Leindler type appears in Ehrhard's inequality (see~\cite{borell1, Latala1}): 
\begin{thmm}
Let $d\gamma(x)= \frac{e^{-|x|^{2}/2}}{(2\pi )^{n/2}} dx$ be the Gaussian measure. And let 
$\Phi(x)=\int_{-\infty}^{x}d\gamma$. Then for any measurable compact sets $A,B \subset \mathbb{R}^{n}$ and any numbers $\lambda, \mu \geq 0$, such that $\lambda+\mu\geq 1$ and $|\lambda-\mu|\leq 1$ we have 
\begin{align}\label{ehh}
\Phi^{-1}(|\lambda A + \mu B|_{\gamma})\geq \lambda \Phi^{-1}(|A|_{\gamma})+\mu\Phi^{-1}(|B|_{\gamma}),
\end{align}
where $|A|_{\gamma}$ denotes Gaussian measure of $A$ i.e., $|A|_{\gamma}=\int_{A}d\gamma$.
\end{thmm}
The inequality initially was stated for convex sets $A$ and $B$. Later it was improved in the sense that only one of them has to be convex and it was conjectured that the inequality is true in general for any measurable sets, and the conjecture was proved recently. 
Ehrhard's inequality is consequence of its functional version:
\begin{thmm}
Let $h,f,g : \mathbb{R}^{n} \to [0,1]$ be functions such that 
\begin{align*}
\Phi^{-1}(h(\lambda x+\mu y))\geq \lambda \Phi^{-1}(f(x))+\mu\Phi^{-1}(g(y)), \quad \text{for all} \quad x,y \in \mathbb{R}^{n},
\end{align*}
where $\lambda, \mu \geq 0$,  $\lambda+\mu\geq 1$ and $|\lambda-\mu|\leq 1$ then 
\begin{align*}
\Phi^{-1}\left( \int_{\mathbb{R}^{n}}hd\gamma\right)\geq \lambda \Phi^{-1}\left(\int_{\mathbb{R}^{n}}fd\gamma\right)+\mu\Phi^{-1}\left(\int_{\mathbb{R}^{n}}gd\gamma\right).
\end{align*}
\end{thmm}
Note that we proved Ehrhard's inequality in Section~\ref{Er}, and the reason the inequality holds was because the function 
\begin{align*}
H(x,y)=\Phi\left( \alpha \Phi^{-1}(x)+\beta \Phi^{-1}(y)\right)
\end{align*}
(where $\Phi(x)=\int_{-\infty}^{x}d\gamma(x)$) satisfies partial differential inequality (\ref{glavnoe}).
\bigskip 
Ehrhard's inequality implies Gaussian isoperimetry, which in turn follows from its integral version:

\begin{corr}
Let $A$ be a Borel set in $\mathbb{R}^{n}$ and let $H$ be an affine halfspace such that $\gamma_{n}(A)=\gamma_{n}(H)=\Phi(a)$ for some $a \in \mathbb{R}$. Then 
\begin{align}\label{giso}
\gamma_{n}(A_{t}) \geq \gamma_{n}(H_{t})=\Phi(a+t)\quad \forall t\geq 0.
\end{align}
where $A_{t} = A+B(t)$, and $B(t)$ is a ball of radius $t$ centered at the origin. 
\end{corr}
Proof follows using Ehrhard's inequality (\ref{ehh}):
\begin{align*}
&|A_{t}|_{\gamma}=\left|\lambda [\lambda^{-1}A]+(1-\lambda)[(1-\lambda)^{-1}tB]\right|_{\gamma}\geq \Phi\left(\lambda \Phi^{-1}(|\lambda^{-1}A|_{\gamma})+(1-\lambda) \Phi^{-1}(|(1-\lambda)^{-1}tB|_{\gamma})\right).
\end{align*}
If we send $\lambda \to 1^{-}$ then $(1-\lambda)^{-1}\Phi^{-1}(|(1-\lambda)^{-1}tB|_{\gamma})\to t$. Indeed, we need to show that  $\lim_{r\to \infty}\frac{1}{r}\Phi^{-1}(|B(r)|_{\gamma})=1$. This follows from the following asymptotic behavior of Gaussian distributions 
\begin{align*}
|B(r)|_{\gamma_{n}} = 1-\frac{|\sigma_{n}|}{(2\pi)^{n/2}}\int_{r}^{\infty}e^{-\frac{r^{2}}{2}}r^{n-1}dr= 1-r^{n-2}e^{-r^{2}/2}\frac{|\sigma_{n}|}{(2\pi)^{n/2}}+o\left( r^{n-2}e^{-r^{2}/2}\right).
\end{align*}
 where $|\sigma_{n}|$ is the measure of the unit sphere in $\mathbb{R}^{n}$, and 
 \begin{align*}
 \Phi(t)=1-\frac{1}{\sqrt{2\pi}}\int_{t}^{\infty}e^{-x^{2}/2}dx=1-\frac{1}{\sqrt{2\pi}}\frac{e^{-t^{2}/2}}{t}+o\left(\frac{e^{-t^{2}/2}}{t}\right).
 \end{align*}
 Thus 
 \begin{align*}
 \Phi^{-1}(s)=(-2\ln(1-s))^{1/2}+o\left((-\ln(1-s))^{1/2} \right)\quad \text{for} \quad r \to 1^{-},
 \end{align*}
 and hence 
 \begin{align*}
\lim_{r \to \infty} \frac{1}{r}\Phi^{-1}(|B(r)|_{\gamma})=\lim_{r \to \infty} \frac{1}{r} \left[-2 \ln (r^{n-2}e^{-r^{2}/2}) \right]^{1/2}=1. 
 \end{align*}

\bigskip
So we obtain the desired result
\begin{align*}
 |A_{t}|_{\gamma}\geq \Phi(\Phi^{-1}(|A|_{\gamma})+t).
\end{align*}

\bigskip 

Infinitisimal version of (\ref{giso}) gives Gaussian isoperimetry 
\begin{cor}
\begin{align*}
|\partial A|_{\gamma} := \lim_{t\to 0} \frac{|A_{t}|_{\gamma}-|A|_{\gamma}}{t} \geq \Phi' (\Phi^{-1}(|A|_{\gamma})).
\end{align*}
\end{cor}

\subsection*{Borell's Gaussian noise ``stability''}
Let $\gamma_{n} = \frac{e^{-|x|^{2}/2}}{(2\pi)^{n/2}}$ be a standard Gaussian measure on $\mathbb{R}^{n}$  and let $\Phi=\int_{-\infty}^{x} d\gamma_{1}$. 
Borell's Gaussian noise ``stability'' (see also~\cite{Mossel3, Mossel1}) states that 
\begin{thmm}
If $A, B$ are measurable subsets of $\mathbb{R}^{n}$. Then if $X=(X_{1},\ldots, X_{n}), Y=(Y_{1},\ldots, Y_{n})$ are independent Gaussian standard random variables, and $p \in (0,1)$ then 
\begin{align*}
\mathbb{P}(X \in A,\; pX+\sqrt{1-p^{2}}\,Y \in B) \leq \mathbb{P} (X_{1} \leq \Phi^{-1}(\gamma_{n}(A)), \;pX_{1}+\sqrt{1-p^{2}}\,Y_{1} \leq \Phi^{-1}(\gamma_{n}(B))).
\end{align*}
\end{thmm}
The functional version of the above inequality can be stated as follows: 
\begin{thmm}
Let $p \in (0,1)$,  $f,g : \mathbb{R}^{n} \to (0,1)$ and let 
\begin{align*}
B(u,v)=\mathbb{P}(X_{1} \leq \Phi^{-1}(u),\; pX_{1}+\sqrt{1-p^{2}}\,Y_{2} \leq  \Phi^{-1}(v)).
\end{align*}
  Then 
\begin{align*}
\int_{\mathbb{R}^{2n}}B\left(f(x), g(px+\sqrt{1-p^{2}}\, y) \right) d\gamma d\gamma \leq B\left(\int_{\mathbb{R}^{n}}f d\gamma, \int_{\mathbb{R}^{n}}g d\gamma \right).
\end{align*}
\end{thmm}
\subsection*{Hypercontractivity}
Let 
\begin{align*}
P_{t}f(x) = \int_{\mathbb{R}^{n}}f(e^{-t}x+\sqrt{1-e^{-2t}}\, y)d\gamma(y)
\end{align*}
 be Ornstein--Uhlenbeck semigroup where  $t\geq 0$. The hypercontractivity for Ornstein--Uhlenbeck semigroup means that 
\begin{thmm}
 Let $p, q >1$ be such that $\frac{q-1}{p-1}\geq  e^{-2t}$. Then 
 \begin{align*}
 \|P_{t}f\|_{L^{p}(d\gamma)}\leq \|f\|_{L^{q}(d\gamma)}.
 \end{align*}
\end{thmm} 

For possible references we refer the reader to \cite{Gross1, Ledoux2014, Paouris1}. 
For the proofs we refer the reader to Section~\ref{BGNS} (for the case $n=1$) and to Section~\ref{gsec1} for  arbitrary $n\geq 1$.

\end{document}